\documentclass[10pt]{amsart}
\usepackage[utf8]{inputenc}
\usepackage[english]{babel}
\usepackage{csquotes}
\usepackage[colorlinks,citecolor=blue,urlcolor=blue,bookmarks=false,hypertexnames=true]{hyperref} 
\usepackage{amsfonts}
\usepackage{mathrsfs}
\usepackage{mathtools}
\numberwithin{equation}{section}
\usepackage{amssymb}
\usepackage{ esint }

\usepackage{aligned-overset}
\usepackage{geometry}

\usepackage[makeroom]{cancel}
\usepackage{enumerate}
\usepackage[final]{graphicx}

\usepackage{subcaption}
\usepackage{amsmath}
\usepackage{latexsym}
\usepackage{amsthm}
\usepackage{amscd}
\usepackage{amsfonts}
\usepackage{ dsfont }
\usepackage{graphics} 
\usepackage{fancyhdr}

\usepackage[dvipsnames]{xcolor}
\theoremstyle{definition}
\newtheorem{defin}{Definition}[section]%chapter
\newtheorem*{defin*}{Definition}

\newtheorem*{oss*}{Remark}

\newtheorem*{es*}{Example}
\newtheorem*{dim*}{}

\theoremstyle{plain}

\newtheorem{teo}[defin]{Theorem}%[chapter]%section
\newtheorem*{teo*}{Theorem}
\newtheorem{cor}[defin]{Corollary}%[chapter]%section
\newtheorem*{cor*}{Corollary}
\newtheorem{lemma}[defin]{Lemma}%[chapter]%section
\newtheorem*{lemma*}{Lemma}
\newtheorem{prop}[defin]{Proposition}
\newtheorem*{prop*}{Proposition}
\allowdisplaybreaks

\newcommand{\N}{\mathbb{N}}

\newcommand{\R}{\mathbb{R}}
\let\P\undefined
\newcommand{\P}{J}
\let\M\undefined
\newcommand{\M}{\mathscr M}

\let\t\undefined
\newcommand{\t}{T_{h,t}^\pm}
\newcommand{\ud}{\,\textnormal{d}}

\def\Jet{\mathcal{J}}
\def\PJet{\mathcal{P}}
\def\G{\mathtt{G}}

\newcommand{\virg}[1]{``#1''}
\newcommand{\bd}{\partial}

\newcommand{\dist}{\textnormal{dist}}
\newcommand{\sd}{\textnormal{sd}^\psi}
\newcommand{\e}{\varepsilon}
\usepackage[maxbibnames=99]{biblatex} %Imports biblatex package
\title[Variational Nonlinear and Nonlocal Curvature Flows]{Variational Nonlinear and Nonlocal Curvature Flows}
\author{Daniele De Gennaro}
\address{Department of Decision Sciences and BIDSA, Bocconi University, Via Roentgen 1, Milano, 20136, Italy}
\email{daniele.degennaro@unibocconi.it}
\date{}

\begin{document}

\begin{abstract}
We prove that  the minimizing movements scheme \'a la Almgren-Taylor-Wang  converges towards level-set solutions to a nonlinear version of nonlocal curvature flows with time-depending forcing term, in the rather general framework of variational curvatures introduced in \cite{ChaMorPon15}. The nonlinearity involved is assumed to satisfy minimal assumptions, namely continuity, monotonicity, and vanishing at zero. Under additional assumptions only on the curvatures involved,   we establish uniqueness for level-set solutions.
\end{abstract}

\maketitle

%\tableofcontents
\section{Introduction}
This paper establishes existence via minimizing movements and uniqueness results for a nonlinear modification of  variational and nonlocal curvature flows in presence of mobility and time-dependent forcing. This nonlinear and nonlocal generalization of the classical mean curvature flow (MCF in short) is defined as follows: given a continuous, non-decreasing function $\G: \mathbb{R} \to \mathbb{R}$ with $\G(0)=0$, we consider the evolution of a family of sets   $t\mapsto E_t$ formally governed by the evolution law
\begin{equation}\label{law evol}
    V(x,t)=\psi(\nu_{E_t}(x))\G\Big( -\kappa(x,E_t) + \mathtt f(t) \Big),\qquad \text{for all }x\in \partial E_t,\  t\ge 0,
\end{equation}
where $\psi$ is an  anisotropy (usually called the mobility), $\nu_{E_t}$ denotes the outer normal vector to $E_t$ and $\mathtt f$ is a forcing term constant in space. In \eqref{law evol}  the curvature $\kappa(\cdot,E)$ denotes a variational curvature, belonging to a class of   generalized nonlocal curvatures introduced  in \cite{ChaMorPon15}.

Generalised curvatures are  functions defined on pairs $(x,E)$, where $E$ is a set of class $C^2$ with compact boundary and $x\in\bd E$, that are non-decreasing with respect to inclusion of sets touching at $x$, continuous w.r.t. $C^2$-convergence of sets, and translation invariant (see conditions (A)-(C) below). %For technical reasons one also requires the curvature of balls to be uniformly bounded from below. 
We will focus on a particular instance of generalized curvatures, namely variational curvature. These curvatures arise as the first variation (in a suitable sense) of perimeter-like functionals, which are called generalized perimeters. A generalized perimeter $J:\M\to [0,+\infty]$ is a translation invariant  functional on the class of measurable sets $\M$, which is insensitive to modifications on negligible sets, finite on $C^2$-sets with compact boundary, lower semicontinuous w.r.t. the $L^1_{loc}$-convergence, and satisfies a submodularity condition: $J(E\cap F)+J(E\cup F)\le J(E)+J(F)$ for every $E,F\in\M$.

The evolution law \eqref{law evol} is relevant even in the specific instance where $\kappa$ is the classical mean curvature, arising as first variation of the perimeter. From a numerical point of view, as suggested e.g. in \cite[Remark 3.5]{ChaNov08}, a truncation of the classical evolution speed $V=-\kappa$
is usually encoded in algorithms for the MCF, which corresponds to choosing $\G(s)=(-M)\vee s\wedge M$ in \eqref{law evol}, for $M>0$ large. Another interesting choice could be $\G(s)=-s^-$ (so that $\G(-\kappa)=-\kappa^+$), which amounts to consider a purely shrinking evolution. Moreover, evolution by powers of the mean curvature   have been previously studied in the smooth or convex setting \cite{And10,Cho87,RisSin,Sch05}  and have been used to prove isoperimetric inequalities \cite{Sch08}, or considered   in the setting of image processing algorithms \cite{AlvGuiLioMor,SapTan}. In particular, in \cite[Section 4.5]{AlvGuiLioMor} it is remarked that the evolution law \eqref{law evol} with $\G(s)=s^{\frac13}$ and $\psi=|\cdot|, \mathtt f=0$ is particularly interesting as it is invariant under affine transformations (isometries and rescalings). See also \cite{DipNovVal} for interesting links between motion by powers of the mean curvature and  a time-fractional Allen-Cahn equation, and \cite{BelKho}, where flat flows solutions to the power (anisotropic) mean curvature flow are studied.

On the other hand, being able to address this study in the framework of generalized curvatures and general nonlinearity $\G$, allows us to prove new results for different geometric flows.
This notion of generalized curvature has been introduced in \cite{ChaMorPon15}  to deal with a wide class of local and nonlocal translation-invariant geometric flows in a unified framework. Some previous contributions can be found in \cite{Car01,Car00,CarLey,Imb,Sle03}.
As detailed in \cite[Section 5]{ChaMorPon15}, some instances of geometric flows driven by variational curvatures are the following: classical \textit{anisotropic MCF} (driven by a suitably smooth and translation invariant anisotropy),  \textit{fractional MCF}, \textit{capacity flows}, and flows driven by the curvature associated to the \textit{regularized pre-Minkowski content}. See also \cite{AzaJimMac} for some extensions. 

Given the definition of variational curvature, and the formal gradient flow structure of the MCF,  one is naturally led to consider the minimizing movements approach, in the spirit of  \cite{AlmTayWan,LucStu}, as a way to prove existence for \eqref{law evol}.  This scheme provides a  discrete-in-time approximation of the evolution law \eqref{law evol} by iteratively solving a variational problem, where the energy to minimize consists of the sum of $J$ and a  suitable dissipation term that penalizes the $L^2$-distance between sets.
In our setting, we will modify the iterative scheme of \cite{ChaMorPon15} (reminiscent of \cite{AlmTayWan,LucStu}),  tailored for the present general setting, by taking into account the nonlinearity in the dissipation term.

We provide here existence via minimizing movements and uniqueness of viscosity solutions to  nonlinear and possibly nonlocal curvature flows  in the presence of continuous time-dependent forcing and mobility, in the form \eqref{law evol}.
Our first main result concerns the instance of \eqref{law evol} where $\kappa$ is a variational curvature. In this case, we show in  Theorem \ref{teo sol viscosa} that  the minimizing movements scheme produces discrete-in-time functions that converge, as the time-step parameter tends to zero, towards a viscosity solution to \eqref{law evol}. 
Subsequently, we establish {uniqueness} for the parabolic Cauchy problem associated with the level set formulation of \eqref{law evol}. This result, presented in Theorem~\ref{thm:comp}, does not require the curvature $\kappa$ to be variational, though it must satisfy specific additional conditions. Remarkably, no further assumptions on $\G$ are needed. In particular, $\kappa$ is required to be either of first-order type or to satisfy a strengthened uniform regularity condition in the second-order case (see conditions (FO) and (C') in Section \ref{sect:uniqueness} for details).
All the relevant examples of generalized curvatures presented above satisfy these assumptions.

The proofs are inspired by the techniques developed in \cite{ChaMorPon15}, coupled with recent insights we developed in \cite{ChaDegMor} (see also \cite{Cha}). In \cite{ChaMorPon15}, the authors prove existence and uniqueness of viscosity solutions to curvature flows of the form $V=-\kappa,$ with $\kappa$ being a generalized curvature (the uniqueness result requires additional assumptions on $\kappa$, the same we will require in the last section). In the specific case of variational curvatures, existence can also be proved by using the minimizing movements scheme, similar to the one sketched above.
The starting observation  is that, under our assumptions on $\G$, if $\kappa$  is a generalized curvature, then $-\G(-\kappa)$ is still a generalized curvature. Therefore, the same viscosity theory of \cite{ChaMorPon15} applies to evolution laws of the form 
\begin{equation}\label{eq:evol-nonlocal}
    V=\G(-\kappa),
\end{equation}
providing existence of viscosity solutions, convergence of the minimizing movements scheme and uniqueness under further assumptions on $\G$ and $\kappa$. Anyhow, when dealing with \eqref{law evol} two problems arise. Firstly, 
it is no longer true in general that if $\kappa$ is a \textit{variational} curvature, then so is $-\G(-\kappa)$. In particular, convergence  of the minimizing movements scheme does not follow immediately from  \cite{ChaMorPon15}.
It is thus interesting to modify the minimizing movements scheme  to account for the nonlinear term, even in the simplified version of \eqref{law evol} given by \eqref{eq:evol-nonlocal}. 
In this regard, nontrivial difficulties arise in the case where $\G$ is bounded from above or below, as some tools heavily employed in the linear setting are no longer available (see e.g. the commonly used reformulation \eqref{pb equi inf}). This issue will be circumvented  by an approximation procedure. One of the main goals of this paper was indeed considering $\G$ with minimal regularity assumptions.

Secondly, the introduction of a time-dependent forcing term and a mobility requires some care. Indeed, the level set formulation for  \eqref{law evol} with time-dependent forcing and mobility does not fall in the framework of \cite{ChaMorPon15}. 
In particular, the proof of the comparison principle needs some careful work. It is  inspired by \cite{ChaMorPon15} with some insights coming from the classical theory of viscosity solutions  (see for instance \cite{Gig-book}). 
 
This work is an extension and an improvement of the unpublished (and unfinished) preprint \cite{ChaCioTho}, where the authors show the convergence of the minimizing movements scheme towards \eqref{eq:evol-nonlocal}, where $\kappa$ the \textit{classical} mean curvature and  $\G$ is a smooth function with polynomial growth.
%Other nontrivial difficulties arise when dealing with generalized perimeters, as many tools from the theory of sets of finite perimeter are no longer available. 
%Finally, notice that classical results yield existence and uniqueness for viscosity solutions to this geometric flow (see e.g. \cite{CheGigGot,EvaSpr,Gig-book}).

To conclude, it would be interesting to study the much more challenging case where the subjacent perimeter is of crystalline type.  In this setting the availability of the viscosity solutions of \cite{GigPoz16,GigPoz18}   and  the development of distribution solutions of \cite{ChaMorNovPon19JAMS,CHaMorNovPon19Anal,ChaMorPon17} may suggest the possibility of a future investigation in this direction. Another interesting instance is the non translation invariant case, and a first  step could be considering  the same setting of  \cite{ChaDegMor}.   {A simplified model would consist in considering a forcing term $\mathtt{f}$ that also depends on the spatial variable. In this case, we expect the convergence result to still hold, by suitably adapting the arguments from the present work and from~\cite{ChaDegMor}. However, the uniqueness result appears to be more delicate as the proof presented here does not carry effortlessly to this instance (see \eqref{eq:problemsf?}). }

The paper is structured as follows. In Section \ref{sect:mms} we introduce some notation and the minimizing movements scheme. Then, in Sections \ref{sect:main_res} we show the convergence of the minimizing movements scheme towards viscosity solutions to \eqref{law evol}. Uniqueness of viscosity solutions to \eqref{law evol}, under additional assumptions on $\kappa$  is the subject of Section \ref{sect:uniqueness}.

\section{The minimizing movements scheme}\label{sect:mms}
\subsection{Preliminaries}
We start introducing some notations.  We will use both $B_r(x)$ and $B(x,r)$ to denote the Euclidean ball in $\R^N$ centered in $x$ and of radius $r$.  If the ball is centered in zero, we simply write $B_r$. We let $\mathscr M$ denote the family of the measurable sets in $\R^N$, and  $E\in C^2$ to say that the set $E$ is of class $C^2$.
In the following, we will always speak about measurable sets and refer to a set as the union of all the points of density $1$ of that set i.e. $E=E^{(1)}.$ Moreover, if not otherwise stated, we implicitly assume that the function spaces considered are defined on $\R^N$,  e.g $L^\infty=L^\infty(\R^N)$. Moreover, we often drop the measure with respect to which we are integrating, if clear from the context.

\begin{defin}%\label{def anisotropy}
   We define  {anisotropy} a function $\psi: \R^{N}\to [0,+\infty)$ which is continuous,  convex, even and positively 1-homogeneous. Moreover, there exists $c_\psi>0$ such that $\forall p\in\R^N$ it holds
\begin{equation}\label{bound velocity}
    \frac 1{c_\psi}|p|\le \psi(p)\le c_\psi |p|.
\end{equation}
\end{defin}
We recall that  the polar function $\psi^\circ$ of an anisotropy $\psi$ is defined by
$$\psi^\circ(v):=\sup_{\psi(\xi)\le 1}\xi\cdot v.   $$
The following identities hold for smooth anisotropies: $\forall v,\xi\in\R^N$ 
\[ \psi(v)\psi^\circ(\xi)\ge v\cdot \xi, \qquad  \psi^\circ(\nabla \psi(v))=v,\qquad \nabla \psi(v)\cdot v=\psi(v). \]
\begin{defin}
    Given an anisotropy $\psi$ and a set $E$, we define the  {$\psi $-distance from $E$} as 
    $$\dist^\psi_E(x)=\inf_{y\in E}\psi^\circ(x-y),$$
    and the  {signed $\psi $-distance} from $E$ as 
    $$\sd_E(x)=\dist^\psi_E(x)-\dist^\psi_{E^c}(x).$$
    For $\delta\in\R$ and $E\in \mathscr M$, we denote
\[ E_\delta=\{ x\in\R^N\ :\ \sd_E(x)\le \delta \}, \]
and use the notation $E_{-\infty}:=\emptyset, E_{+\infty}:=\R^N.$
\end{defin}
Note that \eqref{bound velocity} implies that
\begin{equation}
    \frac 1{c_\psi}\dist_E(x)\le \dist^\psi_E(x)\le c_\psi \dist_E(x),
\end{equation}
where $\dist_E$ denotes the Euclidean distance from the set $E$.

In this section we extend the previous study to nonlocal instances, in the spirit of \cite{ChaMorPon15}. We recall some notation. For any given $E\in C^2$, we consider\footnote{One can slightly generalize this definition by considering sets in $C^{k,\beta}$ with $k\ge 2$, $\beta\in[0,1]$, but for simplicity we consider the $C^2$ case only. } a function $x\mapsto \kappa(x,E)$, defined for $x\in\bd E$, and that we will call (generalized) {curvature} of $E$ at $x$. This function must satisfy the following axioms:
\begin{itemize}
    \item[(A)] Monotonicity: If $E,F\in C^2$ and $x\in \bd E\cap \bd F$ with $E\subseteq F$, then $\kappa(x,E)\ge \kappa(x,F)$;
    \item[(B)] translation invariance: For every $E\in C^2$, $x\in\bd E$ and $y\in\R^N$, it holds $\kappa(x,E)=\kappa(x+y,E+y)$;
    \item[(C)] Continuity: If $E_n\to E$ in $C^2$ and $x_n\in\bd E_n\to x\in\bd E$, then $\kappa(x_n,E_n)\to \kappa(x,E)$.
\end{itemize}
Defining for $x\in\R^N$ and $\rho>0$
\begin{equation}\label{def c}
    \begin{split}
        &\overline c(\rho) = \max_{x\in\bd B_\rho} \max\left\lbrace \kappa(x,B_\rho),  -\kappa(x, B_\rho^c)\right\rbrace,\\
        &\underline c(\rho) = \min_{x\in\bd B_\rho} \min\left\lbrace \kappa(x,B_\rho),  -\kappa(x,B_\rho^c)\right\rbrace,
    \end{split}
\end{equation}
we note that by $(C)$ these functions are continuous in $\rho$. We further require 
\begin{itemize}
    \item[(D)]Curvature of balls: There exists $K>0$ such that $\underline c(\rho)\ge -K>-\infty$.
\end{itemize}
In the following we will focus on the study of the geometric evolution equation
\begin{equation}\label{eq:evol_nonlocal}
V(x,t)=\psi(\nu_{E_t})(x)\G(-\kappa(x,E_t)+\mathtt{f}(t) ),    \quad \text{for $x\in\bd E_t$ and $t>0$},
\end{equation}
starting from an initial bounded set $E_0$ (or an unbounded set with bounded complement), where
  $\psi$ is an anisotropy, $\kappa(\cdot,E_t)$ is a variational curvature in the sense above,  and $\mathtt f$ is a bounded forcing term. Here and in the following, we fix $T>0$ and consider the evolution for $t\in (0,T)$. The functions $\G,\mathtt f$ are required to satisfy the following conditions:
\begin{itemize}
    \item $\G:\R\to \R$  is a continuous, non-decreasing function, with $\G(0)=0$;
    \item $\mathtt f\in C^0_b(\R)$;
\end{itemize}
We then set 
\[ \lim_{s\to -\infty} \G(s)=-a\in [-\infty,0], \qquad  \lim_{s\to +\infty} \G(s)=b\in [0,+\infty]. \]
Consider a function $u:\R^{N}\times [0,+\infty)\to \R$ whose superlevel sets $E_s:=\{ u(\cdot,t)\ge s \}$ evolve according to the nonlinear mean curvature equation \eqref{eq:evol_nonlocal}. By classical computations (see for instance \cite{Gig-book}), the function $u$  satisfies 
\begin{equation}\label{eq:level_set_nonlocal}
    \begin{cases}
        \bd_t u(x,t) - \psi(\nabla u(x,t)) \G(-\kappa(x,\{ u(\cdot,t)\ge u(x,t) \})+\mathtt f(t))=0\\
        u(\cdot,0)=u_0.
    \end{cases}
\end{equation}

Let us recall the notion of viscosity solutions employed in \cite{ChaMorPon15}.
One first introduces a family of auxiliary functions.
\begin{defin}\label{family F}
    Given a curvature $\kappa$ defined as above, we consider a family $\mathcal L$ of functions $\ell\in C^\infty([0,+\infty))$, such that $\ell(0)=\ell'(0)=\ell''(0)=0, \ell(\rho)>0$ for all $\rho$ in a neighborhood of 0, $\ell$ is constant in $[M,+\infty)$ for some $M>0$ (depending on $\ell$), and
    \begin{equation*}
        \lim_{\rho\to 0^+} \ell'(\rho)  \, \G( \overline c(\rho))=0,
    \end{equation*}
where $\overline c$ is as in \eqref{def c}.
\end{defin}
We refer to \cite[Lemma 3.1.3]{Gig-book} for a proof that the the family $\mathcal L$ is not empty. The notion of admissible test function is the following. With a slight abuse of notation, in the   following we will say that a function is spatially constant outside a compact set even if the  value of such constant is time-dependent. 
\begin{defin}\label{def visco sol}
    Let $\hat z=(\hat x,\hat t)\in\R^N\times(0,T)$ and let $A\subseteq (0,T)$ be any open interval containing $\hat t$. We say that $\eta\in C^0(\R^N\times \overline A)$ is admissible at the point $\hat z$ if it is of class $C^2$ in a neighborhood of $\hat z$,  if it is constant out of a compact set, and, in case $\nabla \eta(\hat z)=0$, the following holds: there exists $\ell\in\mathcal L$ and $\omega \in C^\infty([0,+\infty))$ with $\omega'(0)=0,\omega(\rho)>0$ for $\rho>0$ such that 
    \[ 
    |\eta(x,t)-\eta(\hat z)- {\eta_t}(\hat z)(t-\overline t) |\le \ell(|x-\hat x|) + \omega(|t-\hat t|)
    \]
    for all $(x, t)$ in $\R^N\times A$.
\end{defin}
Then, the notion of viscosity solutions employed in \cite{ChaMorPon15} is the following.
\begin{defin}
An upper semicontinuous function $u : \R^N\times [0,T]\to \R$, constant outside a compact set, is a viscosity subsolution of the Cauchy problem \eqref{eq:level_set_nonlocal} if $u(\cdot,0)\le u_0$ and, for all $z:=(x,t)\in\R^N\times[0,T]$ and all $C^\infty$-test functions $\eta $ such that $\eta $ is admissible at $z$ and $u-\eta $ has a maximum at $z$,  the following holds:
\begin{itemize}
    \item[i)] If $\nabla \eta (z)=0$, then
    \begin{equation}\label{eq viscosa degen}
        \eta _t(z)\le 0;
    \end{equation}
    \item[ii)] If $\nabla \eta (z)\neq 0$, then 
    \begin{equation}\label{eq viscosa}
        \bd_t \eta (z) + \psi( \nabla\eta(x,t)) \G(-\kappa(x,\{ \eta(\cdot,t)\ge \eta(z) \})+\mathtt f(t))\le 0.
    \end{equation}
\end{itemize}
A lower semicontinuous function $u : \R^N\times [0,T]\to \R$, constant outside a   compact set, is a viscosity supersolution of the Cauchy problem \eqref{eq:level_set_nonlocal} if $u(\cdot,0)\ge u_0$ and, for all $z:=(x,t)\in\R^N\times[0,T]$ and all $C^\infty$-test functions $\eta $ such that $\eta $ is admissible at $z$ and $u-\eta $ has a minimum at $z$,  the following holds:
\begin{itemize}
    \item[i)] If $\nabla \eta (z)=0$, then $\eta _t(z)\ge 0$,
    \item[ii)] If $\nabla \eta (z)\neq 0$, then $\bd_t \eta (z) + \psi( \nabla\eta(x,t)) \G(-\kappa(x,\{ \eta(\cdot,t)\ge \eta(x,t) \})+\mathtt f(t))\ge 0.$
\end{itemize}
Finally, a function $u$ is a viscosity solution for the Cauchy problem \eqref{eq:level_set_nonlocal} if it is both a subsolution and a supersolution of \eqref{eq:level_set_nonlocal}.
\end{defin}

\begin{oss*}
By classical arguments, one could assume that the maximum of $u-\eta$  is strict in the definition of subsolution above (an analogous remark holds for supersolutions). %Furthermore, in the definition of subsolution we can require $\bd_t \eta(\hat z)/\psi(\nabla \eta(\hat z))>-a$, otherwise the inequality is trivial. Similarly, $\bd_t \eta(\hat z)/\psi(\nabla \eta(\hat z))<b$ for supersolutions.
\end{oss*}

In the rest of the section we will consider a particular instance of generalized curvatures, namely the {variational curvatures} introduced in \cite{ChaMorPon15}. We start by recalling the notion of generalized perimeters.
\begin{defin}\label{def J}
    We will say that a functional $J:\M\to [0,+\infty]$ is a generalized perimeter if it satisfies the following properties: for every $E,E'$ measurable sets and $x\in \R^N$
    \begin{itemize}
        \item[(i)] $J(E)<+\infty$ for every bounded $C^2$-set $E$;
        \item[(ii)] $J(\emptyset)=J(\R^N)=0$;
        \item[(iii)] $J(E)=J(E')$ if $|E\triangle E'|=0$;
        \item[(iv)] $J$ is lower semicontinuous in $L^1_{loc}$;
        \item[(v)] $J$ is submodular, that is
        \begin{equation}\label{eq:submod}
            J(E\cap E') + J(E\cup E')\le J(E) + J(E');
        \end{equation}
        \item[(vi)] $J$ is  translation invariant: for every $E\in C^2$ and $x\in\R^N$ it holds $J(x+E)=J(E)$.
    \end{itemize}
\end{defin}
A generalized perimeter $J$ can be extended to a functional on  $L^1_{loc}(\R^N)$ enforcing  a generalized co-area formula: 
\begin{equation}\label{eq:coarea_J}
    J(u)=\int_{-\infty}^{+\infty} J(\{u\ge s\})\ud s\quad \text{ for every }u\in L^1_{loc}(\R^N).
\end{equation}
It turns out that the functional above is a convex \textit{lsc} functional on $L^1_{loc}(\R^N)$ see \cite{ChaGiaLus}.
\begin{defin}\label{def kappa}
    Given a bounded $C^2$-set $E$ and $x\in \bd E$, we define
    \begin{equation}\label{eq:def_kappa}
            \kappa^+(x,E)  =\inf \left\lbrace \liminf_ {\e\to 0}  \frac{J(E\cup W_\e) - J(E)}{|W_\e\setminus E|} : \overline{W_\e}\overset{\mathcal H}{\to} \{x\}, |W_\e\setminus E|>0  \right\rbrace,
    \end{equation}
    and 
    \[
    \kappa^-(x,E)  =\inf \left\lbrace \liminf_ {\e\to 0}  \frac{J(E) - J(E\setminus W_\e)}{|W_\e\cap E|} : \overline{W_\e}\overset{\mathcal H}{\to} \{x\}, |W_\e\cap E|>0  \right\rbrace,
    \]
    where $\overset{\mathcal H}{\to}$ denotes  Hausdorff convergence. We say that   $\kappa(x,E)$ is the curvature of $E$ at $x$ if $ \kappa^+(x,E) =  \kappa^-(x,E) =: \kappa(x,E).  $
\end{defin}
In the rest of the section we will assume that $\kappa$ exists for all sets of class $C^2$, and furthermore that it satisfies assumption (C) and (D). Assumptions (A) and (B) follow from the assumptions on $J$, furthermore one can prove that the  weak notion of curvature of Definition \ref{def kappa} coincides with the more standard one based on the first variation of the functional $J$, whenever the latter exists (see \cite[Section 4]{ChaMorPon15} for details).

\subsection{The minimizing movements scheme}
We set $g$ as a selection of the set-valued inverse of $\G$, that is $g(x)\in \G^{-1}(x)$ for every $x\in (-a,b)$ and extend it  setting $g=-\infty$ for every $x\le -a$, $g=+\infty$ for every $x\ge b$.  Here, we extended $\G$ to $[-\infty,+\infty]$ setting $\G(\pm\infty)=\lim_{x\to\pm\infty}\G(x).$ We assume also that $g(0)=0$. Note that these definitions  imply $\G\circ g=id$ in $[-a,b].$  Moreover,  $g$ is strictly increasing. 
% \dani{We will denote with $g^*,g_*$ the upper semicontinuous and lower semicontinuous envelope of $g$ (which amounts to choosing the maximal and minimal selection of the inverse $\G^{-1}$)Controllare se serve.}
In the following we will denote for $k\in\N, h>0$
\[ f(kh)= \fint_{kh}^{(k+1)h} \mathtt f(s)\ud s. \]
Given a bounded set $E\in\M$ and $h>0,t\in(0,+\infty)$ we define a functional  on the measurable sets as 
\begin{equation}\label{def mathscr}
\begin{split}
    \mathscr F^E_{h,t}(F)=J(F)+\int_{E\triangle F} \left\lvert g\left( \frac{\sd_E}h  \right)\right\rvert -f([ t/h]h)|F|,
\end{split}
\end{equation} 
where $[\cdot]$ denotes the integer part.
% we have that if $a=+\infty$, then for any $E$ bounded we have that any minimizer of $\mathscr F^E_h$ is a solution of
% \begin{equation}\label{pd equi a inf}
%     \min_F \P(F)+\int_F g\left( \frac{\sd_E}h \right),
% \end{equation}
% as can be seen adding $\int_E g(\sd_E/h)$ to $\mathscr F^E_h$. Analogously, if $b=+\infty$ and $E$ is bounded, any minimizer of $\mathscr F^E_h$ is also a solution of
% \begin{equation}\label{pd equi b inf}
%     \min_F \P(F^c)-\int_{F^c\cap B_{CR}} g\left( \frac{\sd_E}h \right).
% \end{equation}
% Indeed, by Lemma \ref{evolution bounded sets} below, if $E\subseteq B_R$ any minimizer of $\mathscr F^E_h$ is contained in $B_{CR}$. Therefore, since  $E\setminus F=  (F^c\setminus E^c)\cap B_{CR}$ and $F\setminus E=  (E^c\setminus F^c)\cap B_{CR}$, we can add $\int_{E^c\cap B_{CR}}g(\sd_E/h)$ to $\mathscr F^E_h$ we prove the claim above.  
Before proving existence for the functional \ref{def mathscr} we recall the following existence result for a related problem, see \cite[Proposition 6.1]{ChaMorPon15}.
\begin{lemma}\label{lemma existence nonlocal}
    Assume that  $\eta$ is a measurable function satisfying $(-\eta)\vee 0=\eta^-\in L^1(\R^N)$. Then, the problem 
    \begin{equation}\label{def pb J}
        \min\left\lbrace J(F)+\int_F \eta(x)\ud x    \right\rbrace
    \end{equation}
    admits a minimal and a maximal solution (with respect to inclusion). Moreover, if $\eta_1\le \eta_2$ then the minimal (resp. maximal) solution to \eqref{def pb J} with $\eta_1$ replacing $\eta$ contains the minimal (resp. maximal) solution to \eqref{def pb J} with $\eta_2$ replacing $\eta$.
\end{lemma}
% \begin{proof}
% We recall the proof for completeness. Consider the problem
%  \[
%     \min_{u\in L^\infty(\R^N;[0,1])}  J(u)+\int_{\R^N} u(x)\eta(x)     \ud x.
%  \]
%  Since the functional $J$ extended to $L^1_{loc}$ according to \eqref{eq:coarea_J} is convex, one can check that the problem above admits a solution.   Then, note that by coarea and Tonelli's theorem it holds 
% \[
% J(u)+\int_{\R^N} u(x)\eta(x)     \ud x =\int_{0}^1 \left(  J(\{ u>s \}) + \int_{\{ u>s \}}\eta(x)\ud x \right),
% \]
% from which we  deduces that for almost every $s\in[0,1]$, $\{ u>s \}$ is a  solution to \eqref{def pb J}. 
% \end{proof}
We then prove existence of minimizers to $\mathscr F^E_{h,t}$. The proof of the boundedness of minimizers has been taken from \cite{Morini-notes}.
\begin{lemma}\label{lemma existence}
    Let $E\in \mathscr M$ be a bounded set and $h>0,t\in [0,+\infty)$. Then, there exist minimizers of $\mathscr F^E_{h,t}$ and, denoting $E'$ one such minimizer, it has the following properties: it is a bounded set such that (up to negligible sets)
    $$ E_{-ah} \subseteq E'\subseteq E_{bh}.$$
    Moreover, there exist a maximal and a minimal minimizer (with respect to inclusion) of $\mathscr F^E_{h,t}$.
\end{lemma}

\begin{proof}
We fix $h>0$ and $t\in (0,T)$, and $c=f([t/h]h)$.   
Let $n\in\N$ and denote $g_n:=g(\frac{\sd_E}h)\vee -n$ and $\tilde g:=g(\tfrac{\sd_E}h)$. We note that  $g_n^-\in L^1_{loc}$, thus Lemma \ref{lemma existence nonlocal} implies that the functional 
\[
 J(F)+\int_F (g_n-c)
\]
admits a minimal minimizer $E_n$. Since $\int_{E}g_n$ is finite, one can check that $E_n$ minimizes also 
\begin{equation}\label{pb min nonlocal}
     J(F)+\int_{E\triangle F} |g_n|-c|F|.
\end{equation}
Note that $E_n\subseteq E_{n+1}$  by Lemma \ref{lemma existence nonlocal}, therefore  $E_n\to E'=\bigcup_{n\in\N} E_n$ in $L^1_{loc}$. Since $|\tilde g|$ is coercive, there exists $R>0$ such that $|\tilde g|\ge 2\|f\|_{L^\infty(\R)} {+1}$ in $B_R^c$ and $E\subseteq B_R$. Testing \eqref{pb min nonlocal} with $\emptyset$, we deduce  
\[ 
    0\ge J(E_n)+\int_{E_n} (g_n-c)\ge  { (\|f\|_{L^\infty(\R)}+1)} |E_n\setminus B_R|, 
\]
that implies $E_n\subseteq B_R$ for every $n\in\N$.
By semicontinuity and Fatou's lemma we get 
\[  \mathscr F^E_{h,t}(E')\le \lim_{n\to\infty} J(E_n) + \int_{ E_n\triangle E} |g_n|-c|E_n|.  \]
Since $|g_n|\le |\tilde g|$, we conclude that $E'$ is a minimizer of $\mathscr F^E_{h,t}$. 
By classical arguments, one can check that   if $E'_1,E'_2$ are minimizers of $\mathscr F^E_{h,t}$, then so are $E'_1\cap E'_2,E'_1\cup E'_2$, implying the existence of a minimal and a maximal solution (see e.g. \cite[Proposition 6.1]{ChaMorPon15}).

Let now $\tilde E$ denote a minimizer of $\mathscr F^E_{h,t}$. Since $\tilde E$ has finite energy, it is straightforward to check that $|\tilde E|<+\infty$ and  $\sd_E\in [-ah,bh]$ a.e. on $\tilde E\triangle E$. If $b<+\infty$ this clearly  implies that $\tilde E$ is bounded; if $b=+\infty$ we use a different argument. We first prove some preliminary results.
\end{proof}

 {The first one is a comparison principle, in the spirit of \cite{LucStu}.}
\begin{lemma}[Weak comparison principle]\label{comparison principle}
    Fix $h>0,t\in (0,+\infty)$ and assume that $F_1,F_2$ are bounded sets with $F_1\subset\joinrel\subset F_2$. Then, for any two minimizers $E_i$ of $\mathscr F^{F_i}_{h,t}$ for $i=1,2$,  we have $E_1\subseteq E_2$. If, instead, $F_1\subseteq F_2$, then we have that the minimal (respectively, maximal) minimizer of $\mathscr F^{F_1}_{h,t}$ is contained in the minimal (respectively, maximal)  minimizer of $\mathscr F^{F_2}_{h,t}$. 
\end{lemma}

\begin{proof}
Firstly, we assume  $F_1\subset\joinrel\subset F_2$, Testing the minimality of $E_1, E_2$ with their intersection and union, respectively, we obtain
\begin{align*}
   &\P(E_1)+\int_{(E_1\setminus E_2)\setminus F_1} g\left( \frac{\sd_{F_1}}h  \right)+\int_{ (E_1\setminus E_2)\cap F_1} g\left( \frac{\sd_{F_1}}h  \right)\le  \P (E_1\cap E_2) + f([t/h]h)|E_1\setminus E_2| \\
    &\P (E_2)\le \P (E_1\cup E_2)+\int_{(E_1\setminus E_2)\setminus F_2} g\left( \frac{\sd_{F_2}}h  \right)+\int_{(E_1\setminus E_2)\cap F_2} g\left( \frac{\sd_{F_2}}h  \right) - f([t/h]h)|E_1\setminus E_2|.
\end{align*}
Summing the two inequalities above and using the submodularity of $J$ we get 
\begin{multline}\label{ineq comp}
     \int_{(E_1\setminus E_2)\setminus F_1} g\left( \frac{\sd_{F_1}}h  \right) + \int_{(E_1\setminus E_2)\cap F_1} g\left( \frac{\sd_{F_1}}h  \right)\\
     \le  \int_{(E_1\setminus E_2)\cap F_2} g\left( \frac{\sd_{F_2}}h  \right)  + \int_{(E_1\setminus E_2)\setminus F_2} g\left( \frac{\sd_{F_2}}h  \right).
\end{multline}
Assume by contradiction that $|E_1\setminus E_2|>0$. Since $\sd_{F_2}<\sd_{F_1}$ and by the strict monotonicity of $g$, we estimate the \textit{rhs} of \eqref{ineq comp} by
% \[   \int_{(E_1\setminus E_2)\cap F_2} g\left( \frac{\sd_{F_2}}h  \right) - \int_{(E_1\setminus E_2)\cap F_1} g\left( \frac{\sd_{F_1}}h  \right)< \int_{(E_1\setminus E_2)\cap (F_2\setminus F_1)} g\left( \frac{\sd_{F_2}}h  \right) \le 0.\]
\[  \int_{(E_1\setminus E_2)\setminus F_2} g\left( \frac{\sd_{F_2}}h  \right)+ \int_{(E_1\setminus E_2)\cap F_2} g\left( \frac{\sd_{F_2}}h  \right) < \int_{(E_1\setminus E_2)\setminus F_2} g\left( \frac{\sd_{F_1}}h  \right) + \int_{(E_1\setminus E_2)\cap F_1} g\left( \frac{\sd_{F_1}}h  \right)\]
and plug it in \eqref{ineq comp} to reach the desired contradiction. The other cases follow analogously, reasoning by approximation if $F_1\subseteq F_2$.
\end{proof}

\begin{lemma}\label{lemma comparison function}
    Let $c\in \R$. Consider a bounded set  $E\in\M$ and  non-decreasing functions $g_1,g_2 :\R\to \R$ such  that $g_1<g_2$ in $\R\setminus \{0\}$ and $g_1(0)=g_2(0)=0$. Then, if $E_i$ solves
    \[  \min_F \left\lbrace \P(F)+\int_{E\triangle F}\left\lvert g_i(\sd_E(x))\right\rvert\ud x  +c|F|\right\rbrace \]
    for $i=1,2$, we have that $E_2\subseteq E_1$. If $g_1\le g_2$ instead,  an analogous statement holds for the maximal and minimal solutions.
\end{lemma}

\begin{proof}
    Denote $ {\tilde g_i}=g_i\circ \sd_E$ for $i=1,2$  and assume by contradiction that $|E_2\setminus E_1|>0$. Reasoning as in Lemma \ref{comparison principle}, one gets
    \begin{equation*}
        \int_{E_1\triangle E}  |  {\tilde g_1}|+\int_{E_2\triangle E} | {\tilde g_2}|\le  \int_{(E_1\cup E_2)\triangle E}| {\tilde g_1}| + \int_{(E_1\cap E_2)\triangle E} | {\tilde g_2}|.
    \end{equation*}
    Simplifying\footnote{Noting that 
    \begin{align*}
        E_1\triangle E&= ((E_1\setminus E_2)\setminus E)\cup ((E_1\cap E_2)\setminus E) \cup ((E\setminus E_1)\setminus E_2) \cup ((E\cap E_2)\setminus E_1)\\
        (E_1\cup E_2)\triangle E&= (E_2\setminus E_1\setminus E) \cup ((E_1\cap E_2)\setminus E)  \cup ((E_1\setminus E_2)\setminus E)  \cup ((E\setminus E_1)\setminus E_2)\\
        (E_1\cap E_2)\triangle E&=((E_2\cap E_1)\setminus E) \cup ((E\setminus E_1)\setminus E_2) \cup  ((E\cap E_1)\setminus E_2)  \cup ((E\cap E_2)\setminus E_1).
    \end{align*}
    } the above expression and recalling that $ {\tilde g_i}\ge 0$ on $E^c$, $ {\tilde g_i}\le 0$ on $E$, we reach
    \begin{equation*}
        0\le \int_{(E_2\setminus E_1)\setminus E} ( {\tilde g_1}- {\tilde g_2}) +\int_{(E_2\setminus E_1)\cap E} ( {\tilde g_1-\tilde g_2})=\int_{E_2\setminus E_1} ( {\tilde g_1-\tilde g_2}),
    \end{equation*}
    which implies the contradiction. The case $g_1\le g_2$ follows by approximation.
\end{proof}

We can then conclude the proof of the boundedness of minimizers to $\mathscr F^E_{h,t}$. 
\begin{proof}[End of proof of Lemma \ref{lemma existence}]
We prove that any minimizer $\tilde E$ of $\mathscr F^E_{h,t}$ is bounded.
Recall  that $|\tilde E|<+\infty$. 
We assume by contradiction the existence of points $\{x_n\}_{n\in\N}\subseteq \R^N$ of density one for $\tilde E$, with $|x_n|\to +\infty$ as $n\to +\infty$.  {For fixed} $M>0$, since $|g|$ is coercive there exists $R>0$ such that $|\tilde g|\ge M$ in $ B_R^c$. We can assume that $E\subseteq B_{R}$, and, up to extracting an unrelabelled subsequence,  that $|x_n-x_m|> 2R$ for $n\neq m$ and $|x_n|> 3R$ for all $n\in\N$. We note that 
\[ 
M\chi_{B_{2R}^c}< |\tilde g|(\cdot +\tau) \quad \text{ for all }|\tau|\le R.
\]
Let us denote by $E_M$ a minimizer of 
\[
J(F)+ \int_{F\triangle E}  M\chi_{ B_{2R}^c } = J(F) + M|F\setminus B_{2R}|.
\]
By translation invariance $\tilde E+\tau$ minimizes   \eqref{def mathscr} with $|\tilde g|(\cdot +\tau) $ substituting $|\tilde g|$, thus by comparison
\[ 
\tilde E+\tau\subseteq E_M \quad \text{ for all }|\tau|\le R.
\]
In particular, the  disjoint balls $B_R(x_n)$ are all contained (up to negligible sets) in $E_M$. This implies
\[
J(E_M)+ M|E_M\setminus B_{2R}| \ge M|\bigcup_{n\in\N} B_R(x_n) |=+\infty,
\]
a contradiction.
\end{proof}

If $ {\int_{E}g(\sd_E/h)}<+\infty$,  minimizers of $\mathscr F^E_{h,t}$  minimize also the functional
\begin{equation}\label{pb equi inf}
    F\mapsto \P(F)+\int_F  g\left( \tfrac1h\sd_E\right)-f([t/h]h)|F|,
\end{equation}
as can be see adding the (constant term) $ {\int_{E}g(\sd_E/h)}$ to the functional $\mathscr F^E_{h,t}$. 
In the present setting, since  $ {\int_{E}g(\sd_E/h)}$ may be infinite in the case $a<+\infty$, we can not draw this conclusion straightforwardly.  We can nonetheless recover the minimal and the maximal solution to \eqref{incremental pb} by means of a sequence of minimizers of a functional similar to \eqref{pb equi inf}, essentially as in the proof of Lemma \ref{lemma existence}.

% Recall the following definition.

% \begin{defin}
%     A set of finite perimeter $E$ is a $(\Lambda,r_0)-$minimizer of the perimeter ($\Lambda\ge 0, r_0>0$) in an open set $\Omega$ if the following holds: for every ball $B_r(x)\subset \Omega$ , with $0 < r \le r_0$ and any $F\subset \R^N$ such that  $E\triangle F\subset\joinrel\subset  B_r(x)$ we have
%     \[ P(E;B_r(x))\le P(F; B_r(x))+ \Lambda|E\triangle F|. \]
% \end{defin}

% We recall a well-known compactness result, see e.g. \cite[Ch. 21]{Mag-book}.

% \begin{teo} \label{teo conv lambda min}
% Let $\Omega\subset \R^N$ be an open set and $\{E_n\}$ a sequence of locally finite perimeter sets contained in $\Omega$ satisfying the following property: there exists $r_0 > 0$ such that for every $n$,
% $E_n$ is a perimeter $(\Lambda_n, r_0)-$minimizer in $\Omega$, with $\Lambda_n\to \Lambda\in [0, +\infty)$. Then there exist $E\subset \Omega$ of locally finite perimeter and a subsequence  $\{n_k\}$ such that 

% \begin{itemize}
%     \item[(i)] $ E $ is a $(\Lambda, r_0)-$minimizer in $\Omega$;
%     \item[(ii)] $E_{n_k}\to E$ in $L^1_{loc}$;
%     \item[(iii)] $\bd E_{n_k}\to C$ in the Kuratowski sense for some closed set $C$ such that $C\cap \Omega=\bd E\cap \Omega$;
% \end{itemize} 
% %(iv) HN−1 (∂Enk ∩ Ω) ∗* HN−1 (∂E ∩ Ω) weakly* in Ω in the sense of measures.
% \end{teo}

 { For a given bounded set $E\in\M$ and $t\in(0,+\infty)$, we  denote
\begin{equation}\label{incremental pb}
	T_{h,t}^- E=\min\text{argmin}\ \mathscr F_{h,t}^E,\qquad T_{h,t}^+E=\max \text{argmin}\ \mathscr F_{h,t}^E,
\end{equation}
where the minimum and maximum above are made with respect to inclusion. We will often denote $T_{h,t}:=T_{h,t}^-$. 
From the previous results, we deduce this corollary.}

\begin{cor}\label{corollary approx}
    Assume $a<+\infty$. Let $E\in\M$ be a bounded set and $t\in(0,+\infty),h>0$. Then, there exists a sequence of uniformly bounded sets $(E_n)_{n\in\N}$ such that $E_n\nearrow T^-_{h,t} E$  and for any $n\in\N,$ $E_n$ is a minimizer of
    \begin{equation}\label{funct coroll}
         F\mapsto \P(F)+\int_{F}  g\left(  \frac{\sd_E}h  \right)\vee (-n) -f([t/h]h)|F|=:\mathscr F^{E,n}_{h,t}(F). 
    \end{equation}
    Analogously, there exists a sequence of uniformly bounded sets $(E_n)_{n\in\N}$ such that $E_n\searrow T^+_{h,t} E$ in $L^1$ and for any $n\in\N,$ $E_n$ is a solution to
    \begin{equation}\label{funct coroll2}
         \min \left\lbrace \P(F)+\int_{B_R\setminus F}  g\left(  \frac{\sd_E}h  \right)\wedge n -f([t/h]h)|F| : F\subseteq B_R \right\rbrace,
    \end{equation}
    where $\t E\subseteq B_R$.
\end{cor}

\begin{proof}
    We prove the statement for $T^-_{h,t} E$, the other case being analogous. We set $c=f([t/h]h)$,  $g_n:=g(\sd_E/h)\vee (-n)$, and $E'=T_{h,t}^- E$.
    Arguing as in the proof of Lemma \ref{lemma existence}, one builds  a sequence of sets $(E_n)_{n\in\N}$, each being the minimal minimizer of $\mathscr F^{E,n}_{h,t}$, $E_n\subseteq B_R$ for all $n\in\N$ and  $E_n\nearrow \bigcup_{n\in\N} E_n =: \tilde E$. Note that $E'\supseteq E_n$ as $g\le g_n$,   therefore $\tilde E\subseteq E'$  and also $\chi_{E_n\triangle E'}=|\chi_{E_n}-\chi_{E'}|\to \chi_{\tilde E\triangle E'}$ a.e. as $n\to \infty.$  By lower semicontinuity of $J$ and Fatou's lemma we get
    \begin{align*}
        \mathscr F^E_{h,t}(\tilde E) &= \P(\tilde E)  -c|\tilde E| +\int_{\tilde E\triangle E'} |g(\sd_E/h)| = \P(\tilde E) -c|\tilde E|+\int_{\R^N} \liminf_{n\to\infty} \left(|g_n| \chi_{E_n\triangle E}\right)\\
        &\le  \liminf_{n\to\infty} \left(\P(E_n)-c|E_n|+\int_{E_n\triangle E}|g_n| \right).    
    \end{align*}
    Since $E_n$ minimizes $\mathscr F^{E,n}_{h,t}$  we get
    \begin{equation}\label{final eq}
        \mathscr F^E_{h,t}(\tilde E)\le  \liminf_n \left(\P(E')+\int_{E'\triangle E}|g_n|-c|E'|\right)\le \mathscr F^E_{h,t}(E'),
    \end{equation}
    where in the last inequality we used that $|g_n|\le |g|$. Since $E'$ is the minimal minimizer of $\mathscr F^{E}_{h,t}$ we conclude $\tilde E=E'$. The functional \eqref{funct coroll} is obtained from \eqref{def mathscr} adding $\int_{E}g_n(\sd_E/h)$. Finally, the functional in \eqref{funct coroll2} is obtained from functional \eqref{def mathscr} adding the (finite) term $-\int_{B_R\setminus E} g(\sd_E/h)\wedge n$  and restricting the family of competitors.
\end{proof}
We now provide an estimate on the evolution speed of balls. It is interesting to note that, in the isotropic setting ($\psi=\phi=|\cdot|$) and under the assumption of strict monotonicity of $\G$, an explicit evolution law for the radii of evolving balls can be obtained. In our more general case we need to employ the variational proofs of \cite{ChaMorPon15,ChaDegMor}.  By Lemma \ref{lemma existence}, the relevant case is $b=+\infty$.
\begin{lemma}\label{evolution bounded sets}
    Assume $b=+\infty.$ There exists a positive constant $C$ such that, for every $R>0$ and every $t\in (0,+\infty),h>0$ it holds
    \[ T_{h,t}^\pm B_R\subseteq B_{R+Ch}.\]
\end{lemma}
    
\begin{proof} 
It is sufficient to prove the claim for $T_{h,t}^+ B_R$. We fix $h>0$ and set $E=T_{h,t}^+ B_R$ and $\tilde g=g(\sd_{B_R}/h)$. We define
\[ 
\bar \rho =  {\inf}  \{ \rho\in (0,+\infty)  : |E\setminus B_\rho|=0 \}, 
 \]
and note that $\bar\rho<+\infty$ since $E$ is bounded. We can assume \textit{wlog} $\bar\rho>R$. Let $\bar x\in \bd B_{\bar \rho}$ such that $|E\cap B(\bar x,\e)|>0$ for any $\e>0$, and let $\rho>\bar \rho$. Let $\tau=(\tfrac{\rho}{\bar\rho}-1)\bar x$ and note that $B(-\tau,\rho)\supseteq B_{\bar\rho}$ and $\bd B(-\tau,\rho)$ is tangent to $\bd B_{\bar\rho}$ at $\bar x$.

We let for $\e>0$ small $B^\e= B(-(1+\e)\tau,\rho)$ and $W^\e=E\setminus B^\e$. We note that by construction $|W^\e|>0$ and it converges to $\bar x$ in the Hausdorff sense as $\e\to 0$.

Testing the minimality of $E$ against $E\cap B^\e$, we find
\begin{equation}\label{eq:eq0}
     J(E)-\P(B^\e \cap E) \le f([t/h]h)|W_\e| + \int_{B^\e \cap E\triangle B_R} |\tilde g| - \int_{E\triangle B_R} |\tilde g|.
\end{equation}
 We remark that, by the choice of $\bar \rho$ and $\tau$, taking $\e$ small it holds $B_R\subseteq B^\e \cap E$. Therefore, \eqref{eq:eq0} reads
 \[
      J(E)-\P(B^\e \cap E) \le f([t/h]h)|W_\e| + \int_{B^\e \cap E\setminus B_R} |\tilde g| - \int_{E\setminus B_R} |\tilde g|,
 \]
implying
 \begin{equation}\label{eq:eq1}
     J(E)-\P(B^\e \cap E) \le f([t/h]h)|W_\e| - \int_{ (E\setminus B^\e)\setminus B_R} |\tilde g|=f([t/h]h)|W_\e| - \int_{ W^\e} |\tilde g|.
\end{equation}
By submodularity \eqref{eq:submod}, using the definition of $\underline c$ and assumption (D) we conclude
\[
    -K+o_\e(1) \le \|f\|_\infty - \fint_{ W^\e} |\tilde g|\le \|f\|_\infty - \fint_{ W^\e} g(c_\psi(|x|-R)/h ).
\]
Passing to the limit $\e\to0$ we get
\[
K+\|f\|_\infty\ge \liminf_{s\to {c_\psi}(\bar \rho-R)h}g(s),
\]
from which the thesis follows applying $\G$ on both sides.
\end{proof}                              
Note that the previous result implies, in particular, that the discrete evolution starting from an initial bounded set remains bounded in every bounded time interval $(0,T)$. 

We then provide an upper bound  on the evolution speed of balls in the spirit of \cite{ChaMorPon15,ChaDegMor}. We remark that the relevant case is $a=+\infty$ as otherwise Lemma \ref{lemma existence} yields 
\[ \t B_R\supseteq B_{R-ah}. \]

\begin{lemma}	\label{lemma estimates on balls}
    Let $R_0>0$ and $\sigma>1$ be fixed. Assume $a=+\infty$. Then,     there exist a positive constant $c$ such that, if $h>0$ is small enough,  for  all $R\ge R_0$ and $t\in (0,+\infty)$ it holds 
    \begin{equation}\label{decay speed ball}
        \t B_R\supseteq B_{R+\frac h{c_\psi}\G(-\overline{c}(R/\sigma)-\|f\|_\infty )}.
    \end{equation}
\end{lemma}

\begin{proof}
    We prove the result for $E:=T_{h,t}^-B_R$. Take $h$ small enough so that $T_{h,t} B_{\frac 14 R_0}\neq \emptyset$. By translation invariance and taking $h$ small, one can see that\footnote{Indeed, by translation invariance it holds
    \[ T_{h,t} B_{\frac R4}+ B_{\frac 34 R}\subseteq T_{h,t} B_R, \]
    and for $h$ small (depending on $R$) the set $\t B_{R/4}$ is not empty.} $B_{\frac R4}\subseteq E$. We set
    \begin{equation}\label{def bar rho}
        \bar \rho=\sup\{ \rho\in [0,+\infty) : | B_\rho\setminus E|=0 \}\in [\tfrac R4,+\infty),
    \end{equation}
    and note that $\bar \rho<+\infty$ by the boundedness of $E$.
    Assume \textit{wlog} $\bar \rho<R$. Let $\bar x\in\bd B_{\bar \rho}$ be such that $|B(\bar x,\varepsilon)\setminus E|>0$ for any $\varepsilon>0$. Set $\rho\in(0,\bar \rho)$ and $\tau=(1-\rho/\bar\rho)\bar x$ such that $\bd B(\tau,\rho)\cap \bd B_{\bar \rho}=\{\bar x\}$. Setting $B^\e:=((1+\e)\tau,\rho)$, consider the sets 
    \begin{equation*}
        W^\e:=B^\e\setminus E.
    \end{equation*}
    Notice that by construction, for $\e$ small, $W^\e$ has positive measure and it converges to $\{x\}$ as $\e\to 0$ in the Hausdorff sense. Since $E$ minimizes \eqref{pb equi inf} (as $a=+\infty$), we use its minimality to get 
    \[
         \P(\t B_R)-\P(B^\e \cup \t B_R) \le f([t/h]h)|W_\e| + \int_{W^\e} g\left( \frac{\sd_{B_R}}h \right).
    \]
    Dividing by $|W_\e|>0$ the equation above reads
\begin{equation}\label{eq0}
     \frac{\P(\t B_R)-\P(B^\e \cup \t B_R)}{|W_\e|} \le f([t/h]h) + \fint_{W^\e} g\left( \frac{\sd_{B_R}}h \right).
\end{equation}
By submodularity and the definition of variational curvature \eqref{eq:def_kappa} we see that 
\[  
\P(\t B_R)-\P(B^\e \cup \t B_R) \ge \P(B^\e\setminus W_\e)-\P(B^\e )  \ge  |W_\e| \left( -\kappa(\bar x,B^\e)+o_\e(1)  \right),
\]
where $o_\e(1)\to 0 $ as $\e\to0$. 
    % In order to estimate the \textit{lhs} of the equation above, we reason as follows. 
    % By  \eqref{def variaz} with $v=\nabla\phi(x/|x|)$ and by submodularity, we obtain 
    % \begin{equation}\label{split}
    %     \begin{split}
    %         &\int_{\R^N}\nabla\phi\left(\frac x{|x|}\right)\cdot  D\chi_{W^\varepsilon}=\int_{\R^N}\nabla\phi\left(\frac x{|x|}\right) \cdot \left(  D\chi_{B^\e}- D\chi_{B^\e\cap \t B_R} \right)\\
    %         &\le \P(B^\e\cap \t B_R)-\P(B^\e)\le \P(\t B_R)-\P(B^\e \cup \t B_R).
    %     \end{split}
    % \end{equation}
    % Then, using the divergence theorem on the \textit{lhs} of \eqref{split} and combining it with \eqref{eq0} we arrive at
    % \begin{equation}\label{eq 1}
    %     -\fint_{W^\e}\div\,\nabla\phi\left(\frac x{|x|}\right)\le f([t/h]h) + \fint_{W^\e} g\left( \frac{\sd_{B_R}}h \right).
    % \end{equation}
    % By the regularity assumptions on $\phi$ we remark that it holds 
    % \[ |\div \nabla\phi(p)|=|\text{tr}(\nabla^2\phi(p))|\le \frac c{|p|}. \]
    We plug the estimate above in \eqref{eq0} and send $\e\to 0$  to conclude
    \begin{equation*}
        -\overline{c}(\bar \rho)- \|f\|_\infty \le \limsup_{s\to c_\psi(\bar \rho-R)/h } g(s).
    \end{equation*}
    Applying $\G$ to both sides of \eqref{eq0}, we conclude 
    \begin{equation}\label{incremental eq}
        \bar\rho\ge R+\frac h{c_\psi} \G\left(-\overline{c}(\bar\rho)-\|f\|_\infty\right)\ge R+\frac h{c_\psi}\G\left( -\overline{c}(R/4) -\|f\|_\infty \right),
    \end{equation}
    where in the last inequality we recalled that $\bar \rho\ge R/4$. Using again the previous analysis with the bound \eqref{incremental eq}, we show \eqref{decay speed ball} by taking $h$ small enough.
\end{proof}

\subsection{The scheme for unbounded sets}
We now  define the discrete evolution scheme for unbounded sets having compact boundary. Let us introduce the generalized perimeter 
\[
\tilde J(E):=J(E^c)\quad\text{ for all }E\in\M. 
\]
Is is easily checked that $\tilde J$ satisfies all the assumptions of Definition \ref{def J}, and, denoting $\tilde \kappa$ the corresponding curvature, that 
\[   
\tilde \kappa(x,E) = -\kappa(x, E^c).
\]
Therefore, one has the bounds
\begin{equation*}
    \begin{split}
        &\overline c(\rho) = \max_{x\in\bd B_\rho} \max\left\lbrace \tilde \kappa(x,B_\rho),  -\tilde\kappa(x, B_\rho^c)\right\rbrace,\\
        &\underline c(\rho) = \min_{x\in\bd B_\rho} \min\left\lbrace \tilde\kappa(x,B_\rho),  -\tilde\kappa(x,B_\rho^c)\right\rbrace,
    \end{split}
\end{equation*}
where the functions $\overline{c}, \underline{c}$ are defined in \eqref{def c}. 
For every compact set $K$ and $h,t>0$, we let $\tilde T_{h,t}^\pm K$ denote the maximal and the minimal minimizer of $\tilde {\mathscr F}^K_{h,t}$, which corresponds to \eqref{def mathscr} with $\tilde g(s):=-g(-s)$ instead of $g(s)$ and $-f$ instead of $f$.  By changing variable $\tilde F:=F^c$ in \eqref{def mathscr}, we  see that $ (\tilde {T}_{h,t}^- K)^c$ is the maximal solution to
\begin{equation}\label{equiv pb unbounded}
\begin{split}
    \min\left\lbrace \P(\tilde F) +\int_{\tilde F \triangle K^c} \left\lvert  g\left( \sd_{K^c}/h \right)\right\rvert +f([t/h]h)|\tilde F^c| \right\rbrace  .
\end{split}
\end{equation} 
Therefore, for every unbounded set $E$ with compact boundary we define\footnote{To justify this, one can check that if a set $E$ is moving according to \eqref{law evol}, its complement moves according to
\[ V(x,t)=-\psi(\nu_{E^c}(x))\G(\kappa(x,E^c)+\mathtt f)\quad \text{in the direction }\nu_{E^c}, \] from which  the incremental problem follows.}
\begin{equation}\label{def unbounded}
    \t E:= \left(  \tilde T_{h,t}^\mp E^c \right)^c.
\end{equation}
As in the case of compact sets, we set $T_{h,t} E:=T_{h,t}^- E.$  
% Given an unbounded set $E_0$ having compact boundary, we  define the discrete flow $\{E_t\h\}_{t\ge 0}$ as follows: $E\h_t:=E_0$ for $t\in [0,h)$ and
% \[ E\h_t=T_{h,t-h} E\h_{t-h},\quad \forall t\in[h,+\infty). \]
Since $\tilde g$ has the same properties of $g$, one easily checks that analogous results to Lemmas \ref{evolution bounded sets}, \ref{comparison principle} and \ref{lemma estimates on balls} hold also for \eqref{def unbounded}. 

\begin{lemma}\label{comparison principle, unbounded}
Let $t,h>0$. The following statements hold.
\begin{itemize}
    \item Let $F_1\subseteq F_2$ be unbounded sets with compact boundary. Then, $\t F_1\subseteq \t F_2$.
    \item There exists $C>0$ such that for every $R>0,h>0$ it holds $\t B_R^c\supseteq B_{R+Ch}^c.$
    \item Let $R_0>0$ and $\sigma>1$ be fixed. Then, if $a=+\infty$ there exist $c>0$  such that for $h>0$ small enough  and for  all $R\ge R_0$, it holds 
\begin{equation*}
    \t B_R^c\subseteq  B_{R+\frac h{c_\psi}\G(-\sigma\frac cR-\|f\|_\infty)}^c.
\end{equation*}
If instead $a<+\infty$ it holds 
\begin{equation*}
    \t B_R^c\subseteq  B_{R-ah}^c.
\end{equation*}
\end{itemize}
\end{lemma}
Furthermore, Corollary \ref{corollary approx} implies straightforwardly  the following approximation result.
\begin{cor}\label{corollary approx unbounded}
     Set $t,h>0$ and let $E\in\M$ be an unbounded set   with bounded complement.  Then, there exists two sequences of  sets $(E_n)_{n\in\N}, (E_n')_{n\in\N}$ with uniformly bounded complement with the following property. Each $(E_n)^c$ is a minimizer of \eqref{equiv pb unbounded} with $ g\vee (-n)$ substituting $ g$, and $(E_n')^c$ is a minimizer of  \eqref{equiv pb unbounded} with $ g\wedge n$ substituting $ g$. Moreover $E_n\nearrow T_{h,t}^- E$ and $E_n'\searrow T_{h,t}^+ E$.
\end{cor}
We now deduce an equivalent version of \eqref{equiv pb unbounded}, which will be used in the final proof. Let us consider $E$ such that $E^c\subseteq B_R$ and assume $a=+\infty$. Recall that $\t E\supseteq B_{R+Ch}^c$ for some $C>0$ by Lemma \ref{comparison principle, unbounded}.   Adding to the  functional in \eqref{equiv pb unbounded} the term $\int_{B_{R+Ch}\setminus  (T_{h,t}E)^c} g(\sd_{E}/h)$  and restricting the family of competitors, we note that $T_{h,t}^-E$ is the minimal solution to 
\begin{equation}\label{pb equi unbounded}
     \min\left\lbrace \P(\tilde F) +\int_{\tilde F\cap B_{R+ch}}  g\left( \sd_{E}/h \right) + f([t/h]h)|\tilde F^c| : \tilde F^c\subseteq B_{R+ch}\right\rbrace.
\end{equation}
The case $a<+\infty$ needs to be treated by approximation using Corollary \ref{corollary approx unbounded}. Lastly, we state  a comparison principle between bounded and unbounded sets. Its proof follows the one of \cite[Lemma~6.10]{ChaMorPon15}, up to employing  Corollary \ref{corollary approx unbounded}.
\begin{lemma}\label{comparison principle, bounded-unbounded}
    Let $E_1$ be a compact set and let $E_2$ be an open, unbounded set with compact boundary, and such that $E_1\subseteq E_2$. Then, for every $h,t> 0$ it holds $\t E_1\subseteq \t E_2.$
\end{lemma}

\section{Main result}\label{sect:main_res}
We start by introducing the discrete approximation scheme. Given a  continuous function $u_0:\R^N\to \R$ which is constant outside a compact set, we define the transformation 
\begin{equation}\label{def operator}
    T_{h,t} u(x)=\sup\left\lbrace s\in\R\ :\ x\in T_{h,t} \{ u_0\ge s \} \right\rbrace,
\end{equation}
{where the operators $T_{h,t}^\pm$ have been introduced in \eqref{incremental pb},} and  we recall that $T_{h,t}:=T_{h,t}^-$. We then 
 set $u_h(x,t)=u_0(x)$ for $t\in [0,h)$ and  {define}
\begin{equation}\label{def uh}
    u_h(x,t):=\left( T_{h,t-h} u_h(\cdot,t-h)\right)(x).
\end{equation}
By lemmas \ref{comparison principle} and \ref{comparison principle, unbounded}, one can see that the operator $T_{h,t} $ maps functions into functions.  
% Therefore, we can define the following discrete approximation scheme.
% Given an initial uniformly continuous function $u_0$, which is bounded outside a compact set, the discrete approximation scheme is defined as:  $u_h(x,t)=u_0(x)$ for $t\in [0,h)$ and 
% \begin{equation}\label{def uh}
% u_h(x,t):=\left( T_{h,t-h} u_h(\cdot,t-h)\right)(x)\quad \text{ for }t\ge h.
% \end{equation}
The following properties of the operator $T_{h,t}$ hold.
\begin{lemma}
    Given $t,h>0,$ the operator $T_{h,t}$ defined in \eqref{def operator} satisfies the following properties: 
    \begin{itemize}
        \item $T_{h,t}$ is monotone, meaning that $u_0\le v_0$ implies $T_{h,t} u_0 \le T_{h,t} v_0$;
        \item $T_{h,t}$ is translation invariant, as for any $z\in \R^N,$ setting $\tau_z u_0(x):=u_0(x-z)$, it holds $T_{h,t}(\tau_z u_0)=\tau_z (T_{h,t} u_0);$
        \item $T_{h,t}$ commutes with constants, meaning $T_{h,t} (u+c)=(T_{h,t} u)+c$ for every $c\in\R.$
    \end{itemize}
\end{lemma}
\begin{proof}
    The first assertion follows from Lemma \ref{comparison principle} and \ref{comparison principle, unbounded}. The second one follows easily employing the definition \eqref{def operator}, recalling the fact that the functional defined in \eqref{def mathscr} is invariant under translations and that $\{ \tau_z u_0\ge \lambda \}=\{ u_0\ge \lambda \}+z$ for all $\lambda \in\R.$ The last result follows analogously.
\end{proof}
The previous properties satisfied by the operator, in turn, preserve the continuity in space of the initial function.
% \begin{lemma}[Continuity in space]
%     Fix $h>0$. Assume that $u_0\in BV(\R^N)\cap Lip_L(\R^N)$ and it is constant outside a compact set. Then, $u_h$ defined in \eqref{def increm subsol} is $L-$Lipschitz continuous in space for every $t\in [0,+\infty)$.
% \end{lemma}
% \begin{proof}
%     For any $x,y,z\in\R^N$ we have by $L-$Lipschitz continuity of $u_0$ that 
%     \[ u_0(y+z)-L|x-y|\le u_0(x+z)\le u_0(y+z)+L|x-y|. \]
%     Since $T_h$ is a monotone operator, we have 
%     \[ T_h(u_0(\cdot+y)-L|x-y|\le T_h u_0(\cdot+x)\le T_h(u_0(\cdot + y)+L|x-y|), \]
%     and we use the translation invariance and the fact that it commutes with constants to reach
%     \[ (T_h u_0)(\cdot+y)-L|x-y|\le (T_h u_0)(\cdot+x)\le (T_h u_0)(\cdot + y)+L|x-y|). \]
%     Considering the value in $0$ of the previous expression, we recover the $L-$Lipschitz condition for $T_h u_0=u_h(\cdot,h)$. Arguing by induction we conclude.
% \end{proof}
Indeed, assume $u_0$ is uniformly continuous and let $\omega:\R_+\to \R_+$ be an increasing, continuous modulus of continuity for $u_0$. Then, for any $s>s'$ we have 
\[ \{ u>s \}+B_{\omega^{-1}(s-s')}\subseteq \{ u>s' \}, \]
thus, by translation invariance we deduce 
\[ T_{h,t}\{ u>s \}+B_{\omega^{-1}(s-s')}\subseteq T_{h,t}\{ u>s' \}. \]
This inclusion implies that the function $T_{h,t} u_0$ is uniformly continuous in space, with the same modulus of continuity $\omega$ of $u_0$.

The following lemma provides an estimate on the continuity in time of $u_h$. 
\begin{lemma}\label{lemma 6.12 Nonlocal}
    Fix $t,h>0$ and $u_0$ a uniformly continuous function. For all $\lambda\in\R$ it holds 
    \[  T_{h,t}\{ u_h(\cdot,t)> \lambda \}=\{ u_h(\cdot,t+h)>\lambda\} ,\quad  T_{h,t}^+\{ u_h(\cdot,t)\ge  \lambda \}=\{ u_h(\cdot,t+h)\ge \lambda \}. \]
\end{lemma}
\begin{proof}
    Given $\e>0$, by definition it is easy to see that 
    $$\{ T_{h,0} u_0> \lambda+\e\}\subseteq T^\pm_{h,0} \{u_0> \lambda\}\subseteq \{ T_{h,0} u_0> \lambda-\e\}.$$
    Passing to the limit $\e\to 0$, we deduce 
    \[ 
     {\{ u_h(\cdot,h)>  \lambda\}}\subseteq T_{h,0}^\pm \{u_0> \lambda\}\subseteq \{ u_h(\cdot,h)\ge \lambda\}. \]
    Finally, since $u_h(\cdot,h)$ is a continuous function, the equalities $\{ u_h(\cdot,h)> \lambda \}=\text{int}\{u_h(\cdot,h)\ge \lambda \}$ and $\{ u_h(\cdot,h)\ge  \lambda \}=\overline{\{u_h(\cdot,h)\ge \lambda \}}$ hold and we prove the result for $t=h$. The  other cases follow by iteration.
\end{proof}
With the previous results and reasoning exactly  as in \cite[Lemma 6.13]{ChaMorPon15}, we can prove that the functions $u_h$ are uniformly continuous in time.
\begin{lemma}
    For any $\e>0$, there exists $\tau>0$ and $h_0=h_0(\e)>0$ such that for all $|t-t'|\le \tau$ and $h\le h_0$ we have $|u_h(\cdot,t)-u_h(\cdot,t')|\le \e.$
\end{lemma}
Thus, the family $\{u_h\}_{h>0}$ is equicontinuous and uniformly bounded as implied by Lemma \ref{evolution bounded sets}.  By the Ascoli-Arzelà theorem  we can pass to the limit $h\to 0$ (up to subsequences) to conclude that $u_h\to u$  uniformly in any compact in time subset of $\R^N\times [0,+\infty)$, with $u$ being a uniformly continuous function. Moreover, the function $u$ is bounded and constant outside a compact set (as implied by Lemma \ref{evolution bounded sets}). 
\begin{prop}\label{existence by approx}
    Let $T>0$.  Up to a subsequence, the family $\{u_h\}_{h>0}$ defined in \eqref{def uh} converges  uniformly on $\R^N\times [0,T]$ to a uniformly continuous function $u$, which is bounded and constant out of a compact set.
\end{prop}
We can thus state our main result.
\begin{teo}\label{teo sol viscosa}
    The function $u$ defined in Proposition \ref{existence by approx} is a continuous viscosity solution to the Cauchy problem \eqref{eq:level_set_nonlocal}.
\end{teo}
We finally recall the notion of  a level-set solution to the evolution equation \eqref{law evol} (see e.g. \cite{Gig-book}).
\begin{defin}
    Given an initial bounded set $E_0$ (or unbounded set with bounded complement) define an uniformly continuous function $u_0:\R^N\to\R$ such that $\{ u_0>0 \}=E_0.$ Then, setting $u$ as the solution to \eqref{eq:level_set_nonlocal} with initial datum $u_0$ given by Theorem \ref{teo sol viscosa}, we define the level-set solution to the nonlinear mean curvature evolution \eqref{law evol} of $E_0$ as
    \[ E_t:=\{ u(\cdot,t)>0 \}. \]
\end{defin}

\subsection{Proof of the main result}
We start by an estimate on the evolution speed.  For every $r>0$, using the notation of Lemma~\ref{lemma estimates on balls}, we set
\[\hat \kappa(r)=\min\left\lbrace -1, \tfrac1{c_\psi}\G\left(-\overline c(r)-\|f\|_\infty\right) \right\rbrace\]
and, given $r_0>0$, we set $r(t)$ as the unique solution to
\begin{equation}\label{ODE}
    \begin{cases}
        \dot r(t)=\hat \kappa(r(t))\\
        r(0)=r_0.
    \end{cases}
\end{equation}
Note that, in general, the solution $r(t)$ will exist in a finite time interval $[0,T^*(r_0)],$ where $T^*(r_0)$ denotes the extinction time of the solution starting from $r_0$ i.e. the first time $t$ such that $r(t)=0.$

\begin{lemma}\label{lemma evol balls in levelsets}
    Let $u$ be the function given by Proposition \ref{existence by approx} and assume that there exists $\lambda\in\R$ such that $B(x_0,r_0)\subseteq \{ u(\cdot,t_0)>\lambda \}.$ Then, if $a=+\infty$, it holds 
    \[B(x_0,r(t-t_0))\subseteq \{ u(\cdot,t)>\lambda \}\]
    for every $t\le T^*(r_0)+t_0$, where $r(t)$ is the solution to \eqref{ODE} with extinction time $T^*(r_0)$. 
    If instead $a<+\infty$ it holds
    \[B(x_0,r_0-a(t-t_0))\subseteq \{ u(\cdot,t)>\lambda \}\]
    for all $t$ such that $r_0-a(t-t_0)\ge 0$. 
    The same result holds for sublevels substituting superlevel sets.
\end{lemma}

\begin{proof}
    The result in the case $a<+\infty$ follows directly by Lemma \ref{lemma existence}, so we assume $a=+\infty.$  We consider \textit{wlog} $\{ u(\cdot,t_0)>\lambda \}$ bounded, as the other case is analogous. For a fixed $R_0<r_0$, taking $h(R_0)$ small enough, we can ensure that  $B(x_0,R_0)\subseteq \{ u_h(\cdot,t_0)>\lambda \}.$ 
    We then fix $\sigma>1$ and define recursively the radii $R_n$ by  
    $$R_{n+1}=R_n+\tfrac h{c_\psi}\G\left(-\overline{c}(R_n/\sigma)-\|f\|_\infty \right).$$ 
    By Lemmas \ref{comparison principle}, \ref{lemma estimates on balls} and \ref{lemma 6.12 Nonlocal}, we see that $B(x_0,R_{[(t-t_0)/h]+1})\subseteq \{ u(\cdot,t)>\lambda \}$ for every $t\ge t_0$ such that $R_{[(t-t_0)/h]+1}>0.$ Let then $r_\sigma$ be the unique solution to the ODE
    \begin{equation}\label{ODE modif}
    \begin{cases}
        \dot r_\sigma(t)=\hat \kappa(r_\sigma(t)/\sigma)\\
        r_\sigma(0)=R_0.
    \end{cases}
    \end{equation}
    Employing the monotonicity of $\hat \kappa$, if $r_\sigma(t)\le R_n,$ then 
    \begin{align*}
        r_\sigma((n+1)h)&\le R_n+\int_{nh}^{(n-1)h}\hat\kappa\left( \frac{r_\sigma(s)}\sigma \right)\ud s\le  R_n+\int_{nh}^{(n-1)h}\hat\kappa\left( \frac{R_n}\sigma \right)\ud s\\
        &\le  R_n+\int_{nh}^{(n-1)h} \tfrac 1{c_\psi}\G\left( -\overline{c}(R_n/\sigma) -\|f\|_\infty\right)\ud s=R_{n+1}.
    \end{align*}
    Therefore, $B(x_0,r_\sigma(h[(t-t_0)/h]+h)\subseteq \{ u_h(\cdot,t)>\lambda \}$ for $t\ge t_0$ as long as the radius is positive. We conclude sending $h\to 0$, then $R_0\to r_0$ and $\sigma\to 1$.     
\end{proof}

We are now in the position to prove our main result.
 
\begin{proof}[Proof of Theorem \ref{teo sol viscosa}]
Consider $u$ as defined in \eqref{existence by approx}: we show that $u$ is a subsolution, as proving that it is a supersolution is analogous. Let $\eta (x,t)$ be an admissible test function in $\bar z:=(\bar x, \bar t)$ and assume that $(\bar x, \bar t)$ is a strict maximum point for $u-\eta $. Assume furthermore that $u-\eta =0$ in such  point. 

\textbf{Case 1: }We assume that $\nabla\eta  (\bar z)\neq 0$. Firstly, in the case $a<+\infty$ we remark that if $\bd_t \eta/\psi(\nabla \eta(\hat z))\le -a$, then \eqref{eq viscosa} is trivially satisfied, thus we can assume \textit{wlog} that 
\begin{equation}\label{rmk test}
    \frac{\bd_t\eta (\bar z)}{\psi(\nabla \eta(\hat z))}>-a.
\end{equation}
By classical arguments (recalled in \cite{ChaDegMor}) we can assume that each  function $u_{h_k}-\eta$ assumes a local supremum in $B_\rho(\bar z)$ at a point $z_{h_k}=:(x_k,t_k)$ and that $u_{h_k}(z_{h_k})\to u(\bar z)$  as $k\to \infty.$ Moreover, we can assume that $\nabla \eta (z_k)\neq 0$ for  $k$ large enough.

\noindent\textbf{Step 1:} We define a suitable competitor for the minimality of the level sets of $u_h$. By the previous remarks we have that 
\begin{equation}\label{eq 6.19 Nonlocal}
    u_h(x,t)\le \eta (x,t)+c_k
\end{equation}
where $c_k:= u_{h_k}( x_k, t_k)- \eta ( x_k, t_k)  $, with equality if $(x,t)=( x_k,  t_k)$. Let $\sigma>0$ and set 
\begin{equation*}
    \eta _{h_k}^\sigma(x):= \eta (x, t_k)+c_k+\frac{\sigma}2 |x-x_k|^2.
\end{equation*}
Then, for all $x\in\R^N$, 
\[ u_{h_k}(x,t_k)\le \eta ^\sigma_{h_k}(x) \] 
with equality if and only if $x=x_k$. We set $l_k=u_{h_k}(x_k,t_k)=\eta ^\sigma_{h_k}(x_k)$. We fix $\varepsilon>0$, to be chosen later, and define $E_\e^k:=\left\lbrace u_{h_k}(\cdot,t_k)> l_k-\varepsilon  \right\rbrace=T_{h_k,t_k-h_k} \left\lbrace u_{h_k}(\cdot,t_k-h_k) > l_k-\varepsilon  \right\rbrace$\footnote{The choice of working with the open superlevel sets is motivated by our need to employ \eqref{funct coroll}} and
\begin{equation}
    W_\varepsilon^k:= E_\e^k\setminus\left\lbrace \eta ^\sigma_{h_k}(\cdot) >  l_k+\e \right\rbrace  .
\end{equation}
Assume  that $E_\e^k$ is bounded and let us define $E_{\e,n}^{k}$ as the sets constructed by Corollary \ref{corollary approx} where $\left\lbrace u_{h_k}(\cdot,t_k-h_k)\ > l_k-\varepsilon  \right\rbrace,E_\e^k$ substitute  $E,T_{h,t}^-E$ respectively.   We thus have that $E_{\e,n}^{k}\nearrow E_\e^k$  as $n\to \infty$ and that each $E_{\e,n}^{k}$ is the minimal minimizer of a problem  in the form \eqref{pb equi inf}. We define 
\begin{equation}
    W_{\e,n}^k:= E_{\e,n}^{k} \setminus\left\lbrace \eta ^\sigma_{h_k}(\cdot)> l_k+\e,\right\rbrace.
\end{equation}
It is easy to see that, along any subsequence $n(\e)\to \infty$ as $\e \to 0$, it holds $W_{\e,n(\e)}^k\to \{x \}$ as $\e\to 0$ in the Hausdorff sense. Furthermore, we check that  for every $\e,k>0$ there exists $n(\e,k)$ large enough such that $|W_{\e,n}^k|>0$ for all $n\ge n(\e,k)$. Indeed, by the continuity of $\eta^\sigma$ and since $|\nabla \eta(\bar z)|\neq 0$ there exists a positive radius $r$ such that 
\[ ( B(x_k,r)\cap E_\e^k )\subseteq W_\e^k.\]
Since $x_k\in E_\e^k$ and it is an open set, it holds $|W_\e^k|>0$. Recalling that $E_{\e,n}^{k}\to E_\e^k$ in $L^1$, we conclude that $|W_{\e,n}^k|>0$ for all $n=n(\e,k)$ large enough. Note also that, for every fixed $k$, $n(\e,k)\to \infty$ as $\e\to 0$.

By minimality of $E_{\e,n}^{k}$ we have 
\begin{align}
    &\P (E_{\e,n}^{k}) +\int_{E_{\e,n}^{k}} g\left(\tfrac1{h_k} \sd_{\left\lbrace u_{h_k}(\cdot,t_k-h_k) > l_k-\varepsilon  \right\rbrace}(x)\right)\vee (-n) \ud x -f\left(\left[\tfrac t{h_k}\right]h_k\right)|W_{\e,n}^{k}|\nonumber \\
    &\le \P \left(E_{\e,n}^{k} \cap \{ \eta ^\sigma_{h_k}> l_k \}\right) +  \int_{E_{\e,n}^{k} \cap \{ \eta ^\sigma_{h_k}> l_k \} } g\left(\tfrac1{h_k}\sd_{\left\lbrace u_{h_k}(\cdot,t_k-h_k) > l_k-\varepsilon  \right\rbrace}(x)\right)\vee (-n) \ud x. \label{eq 6.21 Nonlocal}
\end{align}
Adding to both sides $\P \left(\{ \eta ^\sigma_{h_k}> l_k \}\cup E_{\e,n}^{k} \right)$ and using the submodularity of $J$, we obtain 
\begin{align*}
    \P (\{ \eta ^\sigma_{h_k}> l_k+\varepsilon \}\cup W_{\e,n}^k) -  \P (\{ \eta ^\sigma_{h_k}> l_k+\varepsilon \})	-f\left(\left[\tfrac t{h_k}\right]h_k\right)|W_{\e,n}^{k}|&\\
    + \int_{ W_{\e,n}^k} g\left(\tfrac1{h_k}\sd_{\left\lbrace u_{h_k}(\cdot,t_k-h_k) > l_k-\varepsilon  \right\rbrace}(x)\right)\vee (-n) \ud x&\le 0.
\end{align*}
Equation \eqref{eq 6.19 Nonlocal} implies  $ \{u_{h_k}(\cdot,t_k-h_k)>  l_k-\varepsilon \}  \subseteq  \{ \eta (\cdot,t_k-h_k)> l_k-c_k-\varepsilon \} $, therefore by monotonicity we get
\begin{equation}
    \begin{split}
        \P (\{ \eta ^\sigma_{h_k}> l_k+\varepsilon \}\cup W_{\e,n}^k) -  \P (\{ \eta ^\sigma_{h_k}> l_k+\varepsilon \})	-f\left(\left[\tfrac t{h_k}\right]h_k\right)|W_{\e,n}^{k}|&\\
        + \int_{ W_{\e,n}^k} g\left(\tfrac1{h_k} \sd_{\{\eta (\cdot,t_k-h_k)> l_k-c_k-\varepsilon \}}(x)\right)\vee (-n) \ud x&\le 0. \label{eq 6.22 Nonlocal}
    \end{split}
\end{equation}
If instead $E_\e^k$ is an unbounded set with compact boundary, we employ \eqref{pb equi unbounded} instead of \eqref{eq 6.21 Nonlocal} to obtain \eqref{eq 6.22 Nonlocal} in the computations above. See \cite{ChaMorPon15,ChaDegMor} for details.

\noindent\textbf{Step 2:} We now estimate the terms appearing in \eqref{eq 6.22 Nonlocal}. We start with the first two   terms  $\P (\{ \eta ^\sigma_{h_k}> l_k+\varepsilon \}\cup W_{\e,n}^k) - \P (\{ \eta ^\sigma_{h_k}> l_k+\varepsilon \})$. By definition of variational curvature, we get 
\begin{equation}\label{eq 6.30 Nonlocal}
       \P (\{ \eta ^\sigma_{h_k}> l_k+\e\}\cup W_{\e,n}^k) - \P (\{ \eta ^\sigma_{h_k}> l_k+\e \} ) \ge |W_{\e,n}^k|\left(  \kappa(x_k,\{ \eta ^\sigma_{h_k}> l_k+\e \}) + o_\e(1) \right),
\end{equation}
The last term in \eqref{eq 6.22 Nonlocal} can be treated as follows. For any $z\in W_\e$, we have
\begin{equation}
    \eta (z,t_k)+c_k+\frac \sigma 2 |z-x_k|^2\le l_k +\varepsilon \label{eq 6.23 Nonlocal}.
\end{equation}
Since, in turn, $\eta (z,t_k)+c_k> l_k-\varepsilon$ it follows that  $ \sigma  |z-x_k|^2<4\varepsilon$ and thus, for $\varepsilon$ small enough,
\begin{equation}
    W_\e^k\subseteq B_{c\sqrt \varepsilon}(x_k). 
\end{equation}
 Therefore, by Hausdorff convergence it holds that for every $\e,k>0$ there exists $n=n(\e,k)$ large enough such that  
 \begin{equation}\label{eq 6.24 Nonlocal}
    W_{\e,n}^k\subseteq B_{2c\sqrt \varepsilon}(x_k). 
\end{equation}
On the other hand, by a Taylor expansion, for every $z\in W_{\e,n}^k$ we have
\begin{equation}\label{eq 6.25 Nonlocal}
    \eta (z,t_k-h_k)=\eta (z,t_k)-h_k\bd_t \eta (z,t_k)+h^2_k\int_0^1 (1-s)\bd^2_{tt}\eta (z,t_k-s h_k)\ud s.
\end{equation}
Then, we consider $y\in\{ \eta (\cdot,t_k-h_k)(y)=l_k-c_k-\varepsilon \}$ being a point of minimal $\psi $-distance from $z$, that is, $ {\psi^\circ(z-y)}=|\sd_{\{ \eta (\cdot,t_k-h_k)(y)>l_k-c_k-\varepsilon \}}(z)|$. One can prove (see \cite[eq. (4.26)]{ChaDegMor}  for details) that
\begin{equation}\label{decay y}
    |z-y|=O(h_k).
\end{equation}
Moreover, it holds (see \cite[eq (6.26)]{ChaMorPon15} for details)
\[ (z-y)\cdot \frac{\nabla \eta(y,t_k-h_k)}{|\nabla \eta(y,t_k-h_k)|}=\pm\psi\left( \frac{\nabla \eta(y,t_k-h_k)}{|\nabla \eta(y,t_k-h_k)|}\right)\dist^\psi_{\{ \eta (\cdot,t_k-h_k)(y)=l_k-c_k-\varepsilon \}}(z),\]
with a \virg{+} if $z\in \{ \eta (\cdot,t_k-h_k)(y)\le l_k-c_k-\varepsilon \}$ and a \virg{-} otherwise. We get
\begin{align}
    \eta (z,t_k-h_k)&=\eta (y,t_k-h_k)+(z-y)\cdot \nabla \eta(y,t_k-h_k)\nonumber\\
    &\ +\int_0^1 (1-s)\left( \nabla^2 \eta(y+s(z-y),t_k-h_k)(z-y)\right)\cdot (z-y) \ud s\nonumber\\
    &=l_k-c_k-\e- \sd_{\{ \eta (\cdot,t_k-h_k)(y)=l_k-c_k-\varepsilon \}}(z)\psi(\nabla \eta(y,t_k-h_k))\nonumber\\
    &\ +\int_0^1 (1-s)\left( \nabla^2 \eta(y+s(z-y),t_k-h_k)(z-y)\right)\cdot (z-y) \ud s\label{eq 6.27 Nonlocal}.
\end{align}   
Note that, in view of \eqref{eq 6.23 Nonlocal}  it holds $|\eta (z,t_k)-\eta (y,t_k)|\le c\varepsilon+ch_k=O(h_k)$, provided $\varepsilon \ll h_k $ and small enough. Thus, using also \eqref{eq 6.24 Nonlocal},\eqref{decay y} we deduce 
\begin{align}
    \frac 1{h_k}  \sd_{\{ \eta (\cdot,t_k-h_k)>l_k-c_k-\varepsilon \}}(z) &\ge \frac {\bd_t \eta (z,t_k)-\frac{2\varepsilon}{h_k}-O(h_k)-O_{h_k}(1)   }{\psi(\nabla\eta (y,t_k-h_k))}   \nonumber\\
    & = \frac {\bd_t \eta (x_k,t_k)+O(\sqrt\varepsilon)-\frac{2\varepsilon}{h_k}-O(h_k)-O_{h_k}(1)   }{\psi(\nabla\eta (x_k,t_k-h_k))+O(\sqrt\varepsilon)+O(h_k)}, \nonumber
\end{align}
and we apply $g$ to both sides to conclude
\begin{equation}\label{eq 6.29 Nonlocal}
     g\left(  \tfrac1{h_k} \sd_{\{ \eta (\cdot,t_k-h_k)>l_k-c_k-\varepsilon \}}(z)\right)\ge g\left(  \frac {\bd_t \eta (x_k,t_k)-O_{h_k}(1)   }{\psi(\nabla\eta (x_k,t_k-h_k))+O(h_k)}  \right)
\end{equation}
\noindent\textbf{Step 4:} We conclude the proof. Combining \eqref{eq 6.22 Nonlocal}, \eqref{eq 6.30 Nonlocal}  and \eqref{eq 6.29 Nonlocal}, we arrive at 
\begin{multline}\label{end eq}
     0\ge |W_{\e,n}^k|    \Big( \kappa(x_k,\{ \eta ^\sigma_{h_k}> l_k+\e \}) + o_\e(1)  -  f\left(\left[\tfrac t{h_k}\right]h_k\right)+  \\
     g\left(  \frac {\bd_t \eta (x_k,t_k)-O_{h_k}(1)   }{\psi(\nabla\eta (x_k,t_k-h_k))+O(h_k)}  \right)\vee (-n)  \Big).
\end{multline}
Choosing $n=n(\e,k)$, we can divide by $|W_{\e,n(\e,k)}^k|>0$ and  apply $\G$ to both sides to get
\begin{multline*}
    \G\left( - \kappa(x_k,\{ \eta ^\sigma_{h_k}> l_k+\e \}) + o_\e(1) + f\left(\left[\tfrac t{h_k}\right]h_k\right)\right) \ge \\ 
     \G\left(   g\left(\frac {\bd_t \eta (x_k,t_k)-O_{h_k}(1)   }{\psi(\nabla\eta (x_k,t_k-h_k))+O(h_k)}\right) \vee (-n(\e,k))\right). 
\end{multline*}
Let us fix $k>0$ and send $\e\to 0$ (thus also $n(\e,k)\to 0$). Thanks to the continuity of $\G$ and recalling also that $W_{\e,n(\e,k)}^k\to \{x\}$ as $\e\to 0$, we let $\e\to 0$ and arrive at
 \[ \G\left( -\kappa(x_k,\{ \eta ^\sigma_{h_k}> l_k+\e \}) +f\left(\left[\tfrac t{h_k}\right]h_k\right) \right)\ge \frac{\bd_t \eta (x_k,t_k)-O_{h_k}(1)}{\psi(\nabla\eta (x_k,t_k))+O(h_k)}, \]
which finally implies the thesis by letting simultaneously $\sigma\to 0$ and $k\to +\infty$. 

\textbf{Case 2:}  We assume $\nabla\eta (\bar x, \bar t)=0$ and prove that $\bd_t \eta (\bar x, \bar t)\le 0$. The proof follows the line of the one in \cite{ChaMorPon15}.  We  focus on the case $a=+\infty$, the other being simpler.

Since $\nabla \eta (\bar z)=0,$ there exist $\ell\in \mathcal L$ and $\omega\in C^\infty(\R)$ with $\omega'(0)=0$ such that 
\[|\eta (x,t)-\eta (\bar z)-\bd_t\eta (\bar z)(t-\bar t)|\le   \ell(|x-\bar x|)+\omega(|t-\bar t|) \]
thus, we can define
\begin{align*}
    &\tilde\eta (x,t)=\bd_t\eta (\bar z)(t-\bar t) + 2 \ell(|x-\bar x|)+2\omega(|t-\bar t|) \\
    &\tilde\eta _k(x,t)=\tilde\eta (x,t)+\frac 1{k(\bar t - t)}.
\end{align*}
We remark that $u-\tilde\eta $ achieves a strict maximum in $\bar z$ and  the local maxima of $u-\tilde\eta _k$ in $\R^N \times [0,\bar t]$ are in points $(x_k,t_k)\to\bar z$ as $k\to \infty$, with $t_n\le \bar t$. From now on, the only difference from \cite{ChaMorPon15} is in the case $x_k= \bar x$ for an (unrelabeled) subsequence. We thus assume $x_k=\bar x$ for all $ k>0$ and define $b_k=\bar t-t_k>0$ and the radii
 \begin{equation*}
    r_k:=\ell^{-1}(a_k b_k),
\end{equation*}
where $a_k\to 0$ must be chosen such that the extinction time for the solution of \eqref{ODE} satisfies $T^*(r_k)\ge \bar t-t_k$, for $k$ large enough.  To show that such a choice for $a_k$ is possible, we set 
\begin{equation}
    \beta(t)=\sup_{0\le s\le t}\hat\kappa(\ell^{-1}(s))\ell'(\ell^{-1}(s)),
\end{equation}
where $\hat \kappa$ is as in \eqref{ODE}. Note that by Definition \ref{family F} it holds $\beta(t)\le \hat\kappa(t)$ for $t$ small, $\beta$ is non decreasing  in $t$ and $g(t)\to 0$ as $t\to 0$. We then have
\begin{align}
    \frac{T^*(r_k)}{b_k}&\ge \frac 1{b_k}\int_{r_k/2}^{r_k}\frac 1{\hat \kappa(s)}\ud s = \frac 1{b_k}\int_{\ell^{-1}(a_k b_k/2)}^{\ell^{-1}(a_k b_k)}\frac 1{\hat \kappa(s)}\ud s\nonumber\\
    &= \frac {a_k}2\fint_{a_k b_k/2}^{a_k b_k}\frac 1{\hat \kappa(\ell^{-1}(r))\ell'(\ell^{-1}(r))}\ud r\ge \frac {a_k}2 \frac1{\beta(b_k)}=2 \label{eq 6.32 Nonlocal},
\end{align}
where in the last equality we chose $a_k:=4\beta(b_k)$ which tends to 0 as $k\to \infty$.   

By definition of $\tilde \eta_k$ it holds 
\begin{align*}
    B(\bar x,r_k)&\subseteq  \{ \tilde\eta _k(\cdot,t_k)\le   \tilde\eta _k(\bar x,t_k)+ 2\ell(r_k) \}\\&\subseteq  \{ u(\cdot,t_k)\le   u(\bar x,t_k)+ 2\ell(r_k)  \},
\end{align*}
by maximality of $u-\tilde\eta _k$  at $z_k$ and since $u(z_k)=\tilde \eta_k(z_k)$. Since the balls $B(\cdot,r_k)$ are not vanishing, by Lemma \ref{lemma evol balls in levelsets} we have
\begin{equation}\label{eq bar x}
     \bar x\in \{ u(\cdot,\bar t)\le   u(\bar x,t_k)+2\ell(r_k) \}. 
\end{equation}
Finally, using again the maximality of $u-\eta $ at $\bar z$, the choice of $r_k$ and  \eqref{eq bar x}, we obtain
\[  \frac{\eta (\bar z)-\eta (\bar x,t_k)}{\bar t-t_k} = \frac{\eta (\bar z)-\eta (\bar x, t_k)}{b_k}\le   \frac{u(\bar z)-u(\bar x,t_k)}{b_k}\le \frac{2\ell(r_k)}{b_k}=2a_k. \]
Passing to the limit $k\to\infty$, we conclude that $\bd_t \eta (\bar z)\le 0$.
\end{proof}

\section{Uniqueness of Viscosity Solutions}\label{sect:uniqueness}

The viscosity theory developed in \cite{ChaMorPon15}  shows {uniqueness} for viscosity solutions to the  Cauchy problem 
\[
\begin{cases}
    \bd_t u(x,t) +|\nabla u(x,t)|\,\kappa(x,\{ u(\cdot,t)\ge u(x,t) \})=0\\
    u(\cdot,0)=u_0,
\end{cases}
\]
which corresponds to    \eqref{eq:level_set_nonlocal} for $\G=id, \psi=|\cdot| $ and $ \mathtt f=0$, under some additional assumptions on the curvature considered. 
In particular, the curvature $\kappa$ must either be of \textit{first order} or satisfy a uniform continuity property (see conditions (FO) and (C') below). Given that the nonlinearity $\G$ is continuous, it follows that if $\kappa$ satisfies the first-order condition, then $-\G(-\kappa)$ also satisfies it. Similarly, assuming $\G$ is uniformly continuous, we deduce that if $\kappa$ satisfies the uniform continuity condition, so does $-\G(-\kappa)$. Consequently,  {uniqueness for continuous viscosity solutions to}
\begin{equation*}
\begin{cases}
    \bd_t u(x,t) -|\nabla u(x,t)|\G(-\kappa(x,\{ u(\cdot,t)\ge u(x,t) \}))=0\\
    u(\cdot,0)=u_0
\end{cases}
\end{equation*}
 {can be deduced from \cite[Theorem 3.5]{ChaMorPon15} (assuming (FO) below) and \cite[Theorem 3.8]{ChaMorPon15} (assuming (C') below and   $\G$    uniformly continuous).  Note however that the curvature $\G(-\kappa)$ is in general not a variational one, thus  the convergence of the minimizing movements scheme does not follow from the results of \cite{ChaMorPon15}. This is instead ensured by Theorem \ref{teo sol viscosa}.}

In this section we detail how one can generalize the results of \cite{ChaMorPon15} to show uniqueness of viscosity solutions to \eqref{eq:level_set_nonlocal}, under some additional assumptions on $\kappa$  (but, quite surprisingly,  not on $\G$). In particular, the major difficulty comes from the presence of a time-dependent term in the operator involving the curvature, which can not be decoupled straightforwardly (because of the presence of the nonlinearity $\G$), see \eqref{eq:level_set_nonlocal}.

\subsection{Setup} We start recalling  notation and some results from \cite{ChaMorPon15}. We start introducing the notion of super/subjets.

\begin{defin}
Let $E\subseteq \R^N$, $x_0\in\partial E$, $p\in\R^N$, and  $X\in Sym(N)$. We say $(p,X)\in \mathcal J^{2,+}_E(x_0)$, the superjet of $E$ at $x_0$, if for every $\delta>0$ there exists a neighborhood $U_\delta$ of $x_0$ such that, for every $x\in E\cap U_\delta$   it holds
\begin{equation}
 (x-x_0)\cdot  p  + \frac12 (X+\delta  I )  (x-x_0)\cdot (x-x_0)  \ge 0.
\end{equation}
Moreover, we say that $(p,X)$ is in the subjet $ \mathcal J^{2,-}_E(x_0)$ of $E$ at $x_0$  if
$(-p,-X)$  is in the superjet $\Jet^{2,+}_{\R^N \setminus E}(x_0)$ of $\R^N \setminus E$ at $x_0$. Finally, 
we say that $(p,X)$ is in the jet $\Jet^{2}_E(x_0)$ of $E$ at $x_0$  if
$(p,X)\in  \Jet^{2,+}_{E}(x_0) \cap \Jet^{2,-}_{E}(x_0)$.
\end{defin}
Analogously, one introduces the notion of parabolic super/subjet. 
\begin{defin}
Let $u:\R^N \times (0, T) \to \R$ be upper semicontinuous
  at $(x,t)$. We say that $(a,p,X)\in \R\times\R^N\times Sym(N)$ is in the parabolic superjet 
$ \PJet^{2,+}u(x,t)$ of $u$ at $(x,t)$,  if  
$$
u( y,s)\leq u(x,t) +  a(s-t)\,+\, p\cdot(y-x)\,+\, \tfrac12(X(y-x))\cdot(y-x)
\,+\,o(|t-s|+|x-y|^2)
$$
 for $(y,s)$ in a neighborhood of $(x,t)$.  If $u$ is lower semicontinuous at $(x,t)$ we can define the parabolic subjet
  $ \PJet^{2,-}u(x,t)$ of $u$ at $(x,t)$ as  $ \PJet^{2,-}u(x,t):=- \PJet^{2,+}(-u)(x,t)$.
\end{defin}

The notion of semijet induces a notion of convergence.

\begin{defin}
Let $E_n\subseteq \R^N$ and $x_0\in\partial E_n$. We say that  $(p_n,X_n)$ are in the superjet $\Jet^{2,+}_{E_n}(x_0)$ uniformly, if 
for every positive $\delta >0$ there exists a neighborhood $U_\delta$ of $x_0$ (independent of $n$) such that, { for all $n\in N$,} 
\begin{equation}%\label{superjeteqins}
  (x-x_0)\cdot p_n + \frac12 (X_n+\delta  I )  (x-x_0)\cdot (x-x_0)  \ge 0 \text{ for every }  x\in E_n \cap U_\delta.
\end{equation}
We say that $(p_n,X_n,E_n)$ converge to $(p,X,E)$  with uniform superjet at $x_0$ if 
$\overline E_n\to \overline E$ in the Hausdorff sense,
the $(p_n,X_n)$'s are in the superjet $\Jet^{2,+}_{E_n}(x_0)$ uniformly and $(p_n,X_n)\to (p, X)$ as $n\to\infty$. 
{Moreover, we say that $(p_n,X_n,E_n)$ converge to $(p,X,E)$  with uniform subjet at $x_0$ if  $(-p_n,-X_n,E^c_n)$ converge to $(-p,-X,E^c)$  with uniform superjet.}
\end{defin}

One can then introduce semicontinuous extensions of $\kappa$.
\begin{defin}
For every $F\subseteq \R^N$ with compact boundary and $(p,X)\in \Jet^{2,+}_{ F}(x)$, we define 
\begin{equation*}\label{defkappal}
\kappa_*(x,p,X,F) :=\ \sup\left\{
\kappa(x,E)\,:\, E \in C^2\,, E\supseteq F\,,
(p,X)\in \Jet^{2,-}_{{E}}(x)
\right\}
\end{equation*}
Analogously, for any $(p,X)\in \Jet^{2,-}_{ F}(x)$ we set
\begin{equation*}\label{defkappau}
\kappa^*(x,p,X,F)\ =\ \inf\left\{
\kappa(x,E)\,:\, E \in C^2\,, \mathring  E\subseteq  F\,,
(p,X)\in \Jet^{2,+}_{ {E}}(x)
\right\}.
\end{equation*}
\end{defin}

As shown in \cite[Lemma 2.8]{ChaMorPon15}, one can prove that $\kappa_*,\kappa^*$ are the l.s.c and the u.s.c. envelope of $\kappa$ with respect to the convergence with uniform superjet and subjet. Noting that 
\[
(-\G(-\kappa))_\ast =-\G(-\kappa_\ast), \quad (-\G(-\kappa))^\ast =-\G(-\kappa^\ast),
\]
one can also show the following equivalent characterization of viscosity solutions.

\begin{lemma}\label{lemma 2.15 nonlocal}
Let $u$ be a viscosity subsolution of~\eqref{eq:level_set_nonlocal} in the sense
of  Definition~\ref{def visco sol}. 
Then, for all $(x,t)$ in $\R^N\times (0,T)$, if $(a,p,X)\in \PJet^{2,+}u(x,t)$,
and $p\neq 0$, it holds
\begin{equation*}\label{eq:sujetsol}
a\,-\, \psi(|p|)\,\G(-\kappa_*\left(x,p,X,\{y: u(y,t)\ge u(x,t)\}+\mathtt f(t)\right) \ \le\ 0.
\end{equation*}
A similar statement holds for supersolutions, with $\PJet^{2,-},\kappa^*$ replacing $\PJet^{2,+},\kappa_*$.
\end{lemma}

\subsection{Proof of the Comparison Principle}
We now show how to adapt the proofs of Theorem~3.5 and Theorem~3.8 of  \cite{ChaMorPon15} to our setting. We will assume that $\kappa$ satisfies assumptions (A)-(D) and either:
% \begin{itemize}
%     \item[(UC)] $G$ is uniformly continuous, with modulus $\omega$,
% \end{itemize}
% and  
\begin{itemize}
    \item[(FO)] For any $\Sigma\in C^2, x\in\bd \Sigma$ and $(p,X),(q,Y)$ in  $\Jet^{2,+}_\Sigma(x),\Jet^{2,-}_\Sigma(x)$ respectively, then 
    \[  \kappa_*(x,p,X,\Sigma)=\kappa^*(x,q,Y,\Sigma) \]
    \item[(C')] Replace (C) by the following.  For every $R>0$ there exists a modulus of continuity $\omega_R$ with the following property. For all $\Sigma \in C^2$, $x\in\partial \Sigma$, such that $\Sigma$ has both an internal and external ball condition of radius $R$ at $x$, and for all $C^2-$diffeomorphism $\Phi:\R^N\to \R^N$, with $\Phi(y)=y$ for $|y-x|\ge 1$, we have  
\[
|\kappa(x,\Sigma) - \kappa(\Phi(x),\Phi(E))|\le \omega_R(\|\Phi - Id\|_{C^2}).
\]
\end{itemize}
If (FO) holds, we say that the curvature $\kappa$ is of first-order, since its relaxation depends only on  the first-order  {elliptic} jet. Otherwise, we say that the curvature $\kappa$ is of second-order. As detailed in \cite{ChaMorPon15}, an instance of first-order curvature is the one associated to the fractional perimeter, while the classical mean curvature is a second-order one satisfying (C').

Assuming (FO), the following comparison between $\kappa_*$ and $\kappa^*$ holds.
\begin{lemma}[Lemma 3.4 in \cite{ChaMorPon15}]\label{lemma 3.4 nonlocal}
    Assume (FO), and let $F,G$ be a closed and an open set respectively, with compact boundary and such that $F\subseteq G$. Let $x\in\bd F, y\in \bd G$ satisfy 
    \[ |x-y|=\dist(\bd F, \bd G). \]
    Then, for all $(p,X)\in \Jet^{2,+}_F(x)$ and $(p,Y)\in \Jet^{2,-}_G(x)$ with $p=x-y$, it holds 
    \[ \kappa_*(x,p,X,F)\ge \kappa^*(y,p,Y,G). \]
\end{lemma}

Assuming instead (C'), we recall the following results from \cite{ChaMorPon15}.
\begin{lemma}\label{lm:compkappa}
Assume (C'). Then,
given $R>0$, there exists a modulus of continuity $\omega_R$ with the following property.
For any $F\in C^2$, $x\in\partial F$, with internal and external ball condition at $x$ of radius $R$, any $(p,X)\in \Jet^{2,+}_F(x)$ with $p\neq 0$,
$|X|/|p|\le 1/R$,
and any $\Phi:\R^N\to \R^N$  diffeomorphism of class {$C^{2}$}, it holds
\begin{equation*}
|\kappa_*(x,p,X,F)- \kappa_*(\Phi(x),D (\psi \circ  \Phi^{-1})(\Phi(x)),
D^2 (\psi \circ \Phi^{-1})(\Phi(x)) ,\Phi(F))|\\
\le \omega_R(\|\Phi - Id\|_{C^{2}})
\end{equation*}
where $\psi(y) =  (y-x)\cdot p   +  \frac{1}{2} X(y-x)\cdot (y-x)$.
The same holds for  $\kappa^*$.
\end{lemma}

\begin{lemma}\label{lm:compC'}
Assume (C').
Let $x\in\R^N$,  $F,G\in C^2$  with $F\subset G\cup \{x\}$
and $\partial F \cap \partial G=\{x\}$.
Let $(p,X) \in \Jet^{2,+}_F(x)$,   
$(p,Y) \in \Jet^{2,-}_G(x)$, with $X\le Y$. Then,
$$
\kappa_*(x,p,X,F) \ge \kappa^*(x,p,Y,G). 
$$
\end{lemma}

Our main result of this section is a comparison principle for sub/supersolutions.
\begin{teo}\label{thm:comp}
Assume either (FO) or (C'). Let $u,v$ be u.s.c and l.s.c functions on $\R^N\times [0,T]$, both constant
outside  a compact set,
a subsolution and a supersolution to \eqref{eq:level_set_nonlocal}, respectively. If $u(\cdot,0) \le v(\cdot,0)$,  
then $u\le v$ in $\R^N\times [0,T]$.
\end{teo}

% We detail the proof in the case when (C') holds, as it is more involved. The reader may check that under (FO) an immediate adaptation of \cite[Theorem 3.5]{ChaMorPon15} holds. We sketch the major steps of the proof, which follows \cite{ChaMorPon15}.

\begin{proof}[Proof assuming (FO)] 
We assume wlog that $u(\cdot,0)<v(\cdot,0)$ and by contradiction that there exists $(\bar x,\bar t)\in\R^N\times (0,T]$ such that $u(\bar x,\bar t)- v(\bar x,\bar t)>0$. Setting $F(t):=\{ u(\cdot,t)\ge u(\bar x,t) \}$ and $G(t):=\{ v(\cdot,t)\ge v(\bar x,t) \}$, it holds $ F(\bar t)\nsubseteq G(\bar t).$
Note that one can perturb the set $F$  {(respectively, the set $G$) so that it satisfies   an internal ball condition (resp. an external ball condition), uniformly in time, and $\chi_F$ is still a subsolution (resp.  $\chi_G$ is still a supersolution).}  {This can be done replacing $F$ by $ F_r$ and $G$ by $\{ \text{sd}_G<-r \}=\text{int}(G_{-r})$ for $r>0$ small so that $F(0)\subseteq G(0)$. } 
Let $\ell\in \mathcal L$. We replace $u,v$ by 
\begin{equation}\label{eq:defuv}
    \begin{split}
        u(x,t)  =\max_{\xi\in\R^N, \tau\in [t-T,t]} \chi_{F(t-\tau)}(x-\xi)-\lambda(\ell(\xi)+\tau^2)\\
        v(x,t)  =\min_{\xi\in\R^N, \tau\in [t-T,t]} \chi_{G(t-\tau)}(x-\xi)+\lambda(\ell(\xi)+\tau^2),
    \end{split}
\end{equation}
where $\lambda$ is a positive parameter, big enough so that $u(\cdot,0)\le v(\cdot,0)$.  {The function $u$ (respectively, the function $v$) is equal to one on $F$ (resp. on $G$), zero outside a compact set, and each superlevel set satisfies an internal ball condition (resp. external ball condition)}, uniformly in time. Furthermore,  {for $\lambda$ large enough (so that the max in \eqref{eq:defuv1} is not reached at $\tau=t$),} $u$ is a subsolution while $v$ is a supersolution in $\R^N\times [2/\sqrt \lambda,T]$.  We refer to \cite{ChaMorPon15} for the proof of these facts.

Let $\alpha,\beta,\e>0$ and set 
\[  \Phi(x,t,y,s):= u (x,t) - v(y,s)  - \alpha \ell(|x-y|) -\beta |t-s|^2 -\e (t+s) . \]
Noticing that $\Phi$ is u.s.c., let $z^\beta=(x^\beta,t^\beta,y^\beta,s^\beta)$ be a maximum point of $\phi$. Note that choosing $\e$ small (depending on $\bar t$), we can assume that the maximum is strictly positive and that $t^\beta,s^\beta$ are strictly positive. Moreover, for $\lambda$ large enough and $\beta\ge \lambda$, one can ensure that $t^\beta,s^\beta>2/\sqrt\lambda.$ \\
%????By classical arguments, $x^\beta,y^\beta$ are uniformly bounded as $\beta\to+\infty,$\footnote{Indeed, note that $|x^\beta-y^\beta|\le M$ independently of $\beta$, and set $R>0$ so that   $u,v$ are constant outside $B_R$. If $|x^\beta|,|y^\beta|>R$ we can consider   $x^\beta+\tau^\beta,y^\beta+\tau^\beta$ for suitable translations $\tau^\beta$, and check that the maximum is unchanged. In the other cases the condition $|x^\beta-y^\beta|\le M$ implies the equiboundedness of $x^\beta,y^\beta$.} 
%thus we find a limit point  $(x_0,t_0,y_0,t_0)$ of $z^\beta$   along a subsequence $\beta_n\to +\infty$. We distinguish two cases.\\
\noindent \textbf{Case 1: }$x^\beta=y^\beta$ along a sequence $\beta_n\to+\infty$. In this case, defining 
\begin{equation}\label{eq:defvarp}
    \begin{split}
        \varphi(x,t) =  v(y^\beta,s^\beta) +\e (t+s^\beta) +\alpha \ell(|x-y^\beta|) +\beta |t-s^\beta|^2   \\
        \psi(y,s) = u(x^\beta,t^\beta) -\e (t^\beta+s ) -\alpha \ell(|x^\beta-y |) -\beta |t^\beta-s |^2,
    \end{split}
\end{equation}
since $u,v$ are a sub- and supersolution respectively, we have 
\[ 0\ge \varphi_t(x^\beta,t^\beta)=2\beta(t^\beta-s^\beta)+\e,\quad 0\le \psi_t(y^\beta,s^\beta)=2\beta(t^\beta-s^\beta) - \e, \]
which yields a contradiction.\\
\noindent \textbf{Case 2: }$x^\beta\neq y^\beta$ for all $\beta$ sufficiently large. 
Note that
\begin{equation*}
    \begin{split}
%(\varphi_t,\nabla\varphi,\nabla^2\varphi)(z^\beta)=
\Bigl(2\beta ( t^\beta - s^\beta) + \e, \alpha f'(| p^\beta| ) \frac{  p^\beta}{| p^\beta |} , X\Bigr) \in \PJet^{2,+} u(  x^\beta,  t^\beta), \\
\Bigl(2\beta ( t^\beta - s^\beta) - \e, \alpha f'(| p^\beta| ) \frac{  p^\beta}{| p^\beta |} , -X\Bigr) \in \PJet^{2,-} v(  y^\beta,  s^\beta),
    \end{split}
\end{equation*}
where  $ p^\beta:= x^\beta- y^\beta$ and $X:=\nabla^2\varphi( x^\beta, t^\beta)$, with $\varphi$ defined in \eqref{eq:defvarp}. Thus, by Lemma~\ref{lemma 2.15 nonlocal}, we have
\begin{equation}\label{eq:subsuperFO}
\begin{split}
2\beta ( t^\beta - s^\beta) + \e -\psi(|p^\beta|)\,  \G\Bigl( -\kappa_*( x^\beta,  \alpha f'(| p^\beta| ) \frac{  p^\beta}{| p^\beta |}, X, \{ u(\cdot, t^\beta ) \ge u( x^\beta,  t^\beta)\})+\mathtt f(t^\beta)\Bigr)\le 0,\\
2\beta ( t^\beta - s^\beta) - \e -\psi(|p^\beta|)\,  \G\Bigl( -\kappa^*( y^\beta,  \alpha f'(| p^\beta| ) \frac{  p^\beta}{| p^\beta |}, -X, \{ v(\cdot, t^\beta) \ge v( x^\beta,  t^\beta)\})+\mathtt f(s^\beta)\Bigr)\le 0.
\end{split}
\end{equation}
Let us denote $\hat p^\beta:= \alpha f'(| p^\beta| ) \frac{  p^\beta}{| p^\beta |}, F^\beta:=\{ u(\cdot, t^\beta) \ge u( x^\beta,  t^\beta)\}, G^\beta :=\{ v(\cdot, t^\beta) \ge v( x^\beta,  t^\beta)\} .$ We then remark that 
\[ \{ u(\cdot, t^\beta) \ge u( x^\beta,  t^\beta)\}+B(0,|{y}^\beta-x^\beta|)\subseteq  \{ v(\cdot, s^\beta) > v(  y^\beta,   s^\beta)\}. \]
Indeed,  if $x\in  \{  u(\cdot, t^\beta) \ge u( x^\beta,  t^\beta)\}$
and $|y-x|< |{y}^\beta-x^\beta|$, since $z^\beta$ is a maximum point for $\Phi$, it holds
\[
v( y^\beta,  s^\beta) - v(y,  s^\beta)
\ \le\ u(  x^\beta,  t^\beta)-u(x,t) +\alpha \ell(|x-y|)-\alpha \ell(| x^\beta-  y^\beta|)
\ <\ 0
\]
so that $y\in \{  v(\cdot, s^\beta) > v(  y^\beta,   s^\beta)\}$. Thus, we can apply Lemma \ref{lemma 3.4 nonlocal} to infer from \eqref{eq:subsuperFO} that 
\begin{equation}\label{eq:contradFO}
    2\e\le  -\psi(|p^\beta|)\,\Bigl(  \G( -\kappa^\beta+\mathtt f(s^\beta)  ) -  \G( -\kappa^\beta +\mathtt f(t^\beta))  \Bigr),
\end{equation}
where we set $\kappa^\beta:=\kappa^*( y^\beta,  \hat p^\beta, -X, G^\beta )$.  Since all the superlevel sets of $u,v$ satisfy a uniform internal, external (respectively) ball condition, and thanks to Lemma \ref{lemma 3.4 nonlocal}, the term $\kappa^\beta$ is bounded  as $\beta\to+\infty$, and so we can assume $(|\kappa^\beta|+\|\mathtt f\|_\infty)\le M.$ Since $\G$ is uniformly continuous in $[-M,M]$, \eqref{eq:contradFO} implies 
\begin{equation*}
    2\e=O(|\mathtt f(s^\beta) -\mathtt f(t^\beta)|)
\end{equation*}
as $\beta\to+\infty$,  a contradiction.
    
\end{proof}

\begin{proof}[Proof assuming (C')]
We assume wlog that $u(\cdot,0)<v(\cdot,0)$ and argue by contradiction. Assume that   there exists $a\in\R$ and $t\in (0,T]$ such that  $F(t):=\{u(\cdot,t) \ge a\}\nsubseteq G(t) :=\{v(\cdot,t) > a\}$.  {As sketched in the previous case,} can assume that
 {$F$ satisfies an internal ball condition while $G$ satisfies an external ball condition,} uniformly in time, and $\chi_F,\chi_G$ are still a sub and supersolution. 
For fixed $\ell\in \mathcal L$ and  $\lambda>0,$ we can replace $u,v$ by 
\begin{equation}\label{eq:defuv1}
    \begin{split}
        u(x,t)  =\max_{\xi\in\R^N, \tau\in [t-T,t]} \chi_{F(t-\tau)}(x-\xi)-\lambda\,(\ell(\xi)+\tau^2)\\
        v(y,s)  =\min_{\xi\in\R^N, \tau\in [s-T,s]} \chi_{G(s-\tau)}(y-\xi)+\lambda\,(\ell(\xi)+\tau^2).
    \end{split}
\end{equation}
Note that it holds  $u(\cdot,0)\le v(\cdot,0)$ for $\lambda$  big enough.
 {The function $u$ (respectively, the function $v$) is equal to one on $F$ (resp. on $G$), zero outside a compact set, and each superlevel set satisfies an internal ball condition (resp. external ball condition)}, uniformly in time. Furthermore,  {for $\lambda$ large enough (so that the max in \eqref{eq:defuv1} is not reached at $\tau=t$),} $u$ is a subsolution while $v$ is a supersolution in $\R^N\times [2/\sqrt \lambda,T]$.   {In the following, we} omit the dependence on $\lambda,$ as it will be a fixed parameter.  For $\alpha,\beta,\e>0$ and $\N\ni \beta\ge \lambda$, we define
$$
\Phi(x,t,y,s):= u (x,t) - v(y,s) -\e (t+s) - \alpha \ell(|x-y|) -\beta |t-s|^2,
$$ 
which is semiconvex. For $\e>0$ small and $\alpha,\beta$ large enough the function $\Phi$ admits a positive maximum 
at some $(x^\beta,t^\beta,y^\beta,s^\beta) \in \R^N\times [0,T]\times \R^N\times [0,T]$ with $t^\beta,\, s^\beta>0$. Note also that 
$$
|t^\beta-s^\beta|\to 0 \quad \text{ as }\beta\to+\infty.
$$ 
Since $u,v$ are  {constant} outside a compact set, and by translation invariance  {it} is not difficult to see that $x^\beta,y^\beta$ admit cluster points $x_0,y_0$ as $\beta\to+\infty$ (see for instance \cite[page~14]{Morini-notes}). We thus assume wlog that $(x^\beta,y^\beta)\to(x_0,y_0)$ as $\beta\to +\infty$. If $x^\beta=y^\beta$ infinitely often,
one can conclude considering $\varphi,\psi$ defined in  \eqref{eq:defvarp} (see the previous proof and   \cite{ChaMorPon15}). Thus, we assume $x^\beta\neq y^\beta$ for all $\beta$.
 One can also assume that $\ell(|x^\beta-y^\beta|)<1$  {(taking $\lambda$ large)} and check that 
 \begin{equation}\label{eq:umin1}
 	u(x^\beta,t^\beta) <1.
 \end{equation}
  {Indeed, $D u(x^\beta,t^\beta)=D\ell(|x^\beta-y^\beta|)\neq 0$, while  $u(x,t) = 1$ if and only if $x\in F(t)$, but on $F(t)$ it holds $Du=0$.}\\
\noindent\textbf{Step 1:} In this step we provide  estimates for the final argument.  {The constructions are essentially the same introduced in \cite{ChaMorPon15}, which we recall for the reader's convenience.}
We fix $\beta$ and  omit the dependence on it of the approximating parameters.

Let $q: [0,+\infty]\to [0,1]$ be a smooth, nondecreasing, function with $q(r)= r^4$ for $r<1/2$
and $q(r)=1$ for $r>3/2$. For $\rho>0$  {we define}
$$
\Phi_{  \rho}(x,t,y,s):= \Phi(x,t,y,s) - \rho [q(|x-x^\beta|)+ q(|y-y^\beta|)+q(|t-t^\beta|)+q(|s-s^\beta|)],
$$
so that $(x^\beta,t^\beta,y^\beta,s^\beta)$ is a strict maximum of $\Phi_{  \rho}$. 
Let $\eta:\R^N\to \R$ be a smooth cut-off function, with compact support and equal to one in a neighborhood $U$ of the origin. 
 {For every $\Delta:=(\zeta_u,\tau_u,\zeta_v,\tau_v)\in \R^N\times \R\times \R^N\times \R$, the function 
\[
    \Phi_\rho(x,t,y,s)-\Big( \eta(x-x^\beta)\, (\zeta_u,\tau_u)\cdot(x,t)+  \eta(y-y^\beta)\, (\zeta_v,\tau_v)\cdot(y,s)   \Big)
\]
is maximized at some $(x_\Delta,t_\Delta,y_\Delta,s_\Delta)$ converging to $(x^\beta,t^\beta,y^\beta,s^\beta)$ as $|\Delta|\to0$. Therefore,  by Jensen's Lemma \cite[Lemma A.3]{CraIshLio}}
we may assume that for every $\delta>0$ sufficiently small there exists 
$\Delta _{\rho, \delta}:=(\zeta^{\rho, \delta}_u, h^{\rho, \delta}_u ,\zeta^{\rho, \delta}_v, h ^{\rho, \delta}_v)$, {with $|\Delta _{\rho, \delta}|\le \delta$}, such that the function
\begin{equation*}
\Phi_{  \rho, \delta}(x,t,y,s):=\Phi_{\rho}(x,t,y,s) \\ 
- \Big(\eta(x - x^\beta)  (\xi ^{\rho, \delta}_u, h ^{\rho, \delta}_u)\cdot (x,t) + \eta(y - y^\beta)  (\xi ^{\rho, \delta}_v,h ^{\rho, \delta}_v)\cdot  (y,s)  \Big)
\end{equation*}
attains a maximum at some $z _{\rho, \delta}:=(x _{\rho, \delta}, t _{\rho, \delta}, y _{\rho, \delta}, s _{\rho, \delta})$
 where $\Phi_{\delta, \rho}$ is twice differentiable and such that  $x _{\rho, \delta}-x^\beta$, $y _{\rho, \delta}-y^\beta \in U$ and $t _{\rho, \delta},\, s _{\rho, \delta} >0$.  Moreover, 
 \begin{equation}
z _{\rho, \delta} \to (x^\beta, t^\beta, y^\beta, s^\beta) \qquad \text{ as }  \delta \to 0. 
\end{equation}

Notice that since $\Phi_{\rho}$ is twice differentiable  at $z _{\rho, \delta}$ it follows that also $u,v$ are twice differentiable at  $(x _{\rho, \delta}, t _{\rho, \delta})$ and $(y _{\rho, \delta}, s _{\rho, \delta})$, respectively.

 {
Let   $ \tau^{\rho, \delta}_u\in\R$ (resp.  $\tau^{\rho, \delta}_v \in \R$) be the maximizing  (resp. minimizing) $\tau$  in  \eqref{eq:defuv1}  corresponding to the point $(x _{\rho, \delta}, t _{\rho, \delta})$ (resp. $(y _{\rho, \delta},s _{\rho, \delta})$). 
}
Setting
\begin{equation*}
\begin{split}
    &\tilde u(x,t) := \max_{\xi\in \R^N}  \left\{ \chi_{F(t- \tau^{\rho, \delta}_u)} (x-\xi) - \lambda  \ell(|\xi|)\right\}  - \lambda   (\tau^{\rho, \delta}_u)^2 \\ 
&\tilde v(y,s):= \min_{\xi\in \R^N} \left\{ \chi_{G(s- \tau^{\rho, \delta}_v)}(y - \xi) +  \lambda \ell(|\xi|) \right\} + \lambda   (\tau^{\rho, \delta}_v)^2, 
\end{split}
\end{equation*}
we note that 
\begin{equation}\label{eq:tildeu}
\begin{split}
& u \ge \tilde u,
\qquad   u (x _{\rho, \delta},t _{\rho, \delta}) =  \tilde u(x _{\rho, \delta},t _{\rho, \delta}),\\
&v \le \tilde v,  
\qquad v (y _{\rho, \delta},s _{\rho, \delta}) = \tilde v(y _{\rho, \delta},s _{\rho, \delta}).      
\end{split}
\end{equation}
Set now
\begin{align*}
\hat u(x,t):= \tilde u(x,t) -\rho \left(q(|x-x _{\rho, \delta}|) +q(|x-x^\beta|) + q(|t-t^\beta|)\right) - \eta(x - x^\beta)  (\xi ^{\rho, \delta}_u, h ^{\rho, \delta}_u)\cdot (x,t) ,\\
\hat v(y,s):=  \tilde v(y,s) +\rho\left( q(|y-y _{\rho, \delta}|)+ q(|y-y^\beta|) +q(|s-s^\beta|) \right) + \eta(y - y^\beta)  (\xi ^{\rho, \delta}_v,h ^{\rho, \delta}_v)\cdot  (y,s) .
\end{align*}
Then, the function
\[
 \hat u(x,t) -  \hat v(y,s)
 -\e(t+s)- \alpha \ell(|x-y|) -\beta |t-s|^2
\]
has a maximum at  $z _{\rho, \delta}$, which is 
strict with respect to the spatial variables. Thus
\begin{equation*}
 \hat F_{\rho, \delta}(t):=\{ \hat u(\cdot,t) \ge 
\hat u(x _{\rho, \delta},t _{\rho, \delta})\}, 
\qquad
 \hat G_{\rho,\delta}(s):=\{   \hat v(\cdot,s) > \hat v(y _{\rho, \delta},s _{\rho, \delta}) \}. 
\end{equation*}
satisfy $\hat F_{\rho,\delta}(t _{\rho, \delta})\subseteq \hat G_{\rho,\delta}(s _{\rho, \delta})$ and moreover $x _{\rho, \delta}\in \hat F_{\rho,\delta}(t _{\rho, \delta})$ and $y _{\rho, \delta}\in \hat G_{\rho,\delta}(t _{\rho, \delta})$ are the only points realizing the distance between $\hat F_{\rho,\delta}(t _{\rho, \delta})$ and $\hat G_{\rho,\delta}(t _{\rho, \delta})$.
In particular, $\hat F_{\rho,\delta}(t _{\rho, \delta})$  {(respectively,   $\hat G_{\rho,\delta}(t _{\rho, \delta})$) satisfies an  external ball condition  (resp. internal ball condition) of radius $|x^\beta-y^\beta|>0$.}
We observe that at the maximum point,
$$
\Big| |D \hat u(x_{\rho,\delta},t_{\rho,\delta})|-\alpha \ell'(|x^\beta-y^\beta|)  \Big|=\omega(\rho,\delta)
$$ 
where $\omega\to 0$ as its arguments tend to 0, thus since $\ell'(|x^\beta-y^\beta|)\neq 0$,  the term $ |D \hat u(x_{\rho,\delta},t_{\rho,\delta})|$ is bounded below for $\rho,\delta$ small. In addition, the function
$\hat u$ is semiconvex, hence $\hat{F}_{\rho,\delta}(t_{\rho,\delta})$ has an
interior ball condition at $x_{\rho,\delta}$ with a radius depending on $\lambda$ only, thus independent
on $\rho,\delta$, if small enough, and $\beta$. Analogously, 
$\hat G_{\rho,\delta}(s _{\rho, \delta})$ has an exterior  ball condition at
$y_{\rho,\delta}$ with a radius depending on $\lambda$ only. 

Set 
\[
\breve \Phi_{  \rho, \delta}(x,t,y,s):= \Phi_{\rho, \delta}(x,t,y,s)+\alpha   \ell(|x-y|) +\beta  |t-s|^2
\]
and
\begin{align*}%\label{super1}
& ( \breve a_{\rho,\delta}, \breve p_{\rho,\delta}, \breve X_{\rho,\delta}) 
:= ( \partial_t \breve\Phi_{  \rho, \delta} (z _{\rho, \delta}) , D_x \breve\Phi_{  \rho, \delta} (z _{\rho, \delta}), D^2_x \breve\Phi_{  \rho, \delta} (z _{\rho, \delta}) ),
\\ 
%\label{super2}
& (\breve b_{\rho,\delta}, \breve q_{\rho,\delta},  \breve Y_{\rho,\delta}) :=  ( \partial_s \breve\Phi_{  \rho, \delta} (z _{\rho, \delta}) , D_y \breve\Phi_{  \rho, \delta} (z _{\rho, \delta}), D^2_y \breve\Phi_{  \rho, \delta} (z _{\rho, \delta}) ).
\end{align*}
 { Then, recalling \eqref{eq:tildeu},  we observe that the superjet $( \breve a_{\rho,\delta}, \breve p_{\rho,\delta},  \breve X_{\rho,\delta})$ of
\[
u(x,t)-
\rho [q(|x-x^\beta|) + q(|t-t^\beta|)] - \eta(x - x^\beta)   (\xi ^{\rho, \delta}_u, h ^{\rho, \delta}_u)\cdot (x,t) 
\]
at $(x_{\rho,\delta},t_{\rho,\delta})$ is also a superjet for $\hat{u}(x,t)$
at the same point. Since $\hat u(x,t)\geq \hat u (x _{\rho, \delta},t _{\rho, \delta}) \chi_{ \hat F_{\rho,\delta}(t)}(x)$ and $x_{\rho,\delta}\in \hat F_{\rho,\delta}(t_{\rho,\delta})$,    we have}
\begin{align}
 ( \breve a_{\rho,\delta}, \breve p_{\rho,\delta},  \breve X_{\rho,\delta}) \in \PJet^{2,+}_{\hat u (x _{\rho, \delta},t _{\rho, \delta}) \chi_{ \hat F_{\rho,\delta}} }(x _{\rho, \delta},t _{\rho, \delta}), \label{parajet}\\
 ( \breve b_{\rho,\delta},  \breve q_{\rho,\delta},  \breve Y_{\rho,\delta}) \in \PJet^{2,-}_{\hat v (y _{\rho, \delta},s _{\rho, \delta}) \chi_{  \hat G_{\rho,\delta}} }(y _{\rho, \delta},s _{\rho, \delta}).\label{parajet2}
\end{align}
Since $z_{\rho,\delta}$ is a maximum of $\Phi_{\rho,\delta}$,
\begin{equation}\label{eq:relparjet}
 \breve a_{\rho,\delta} -  \breve b_{\rho,\delta} = 2\e, \qquad  \breve p_{\rho,\delta} =  \breve q_{\rho,\delta}, \quad  
 \breve X_{\rho,\delta} \le  \breve Y_{\rho,\delta}.
\end{equation} 
By construction, $\breve \Phi_{\rho,\delta}$ is also semiconvex, so 
that $\breve X_{\rho,\delta} \ge -cI$, $\breve Y_{\rho,\delta}\le
cI$ for a constant $c$ that does not depend on $\rho,\delta$.

%As detailed in \cite{ChaMorPon15}, one can then build a $C^2$ diffeomorphism $\Psi_{\rho,\delta}$ with the following properties: $\Psi_{\rho,\delta}(x_{\rho,\delta})=x_{\rho,\delta}$, it is a constant (small) translation outside a neighborhood of $x^\beta$, it converges $C^2$ to the identity as $\rho,\delta\to 0$.

 {We then let
\[
c_{\rho,\delta}(x,t) = \tilde u(x,t)+(\hat u(x_{\rho,\delta},t_{\rho,\delta}) - \hat u(x,t) )
\]
and note that $c_{\rho,\delta}\to u(x^\beta,t^\beta) $ uniformly  as $\rho,\delta\to 0 $. Thus, thanks to \eqref{eq:umin1}  we can assume $c_{\rho,\delta}<1$. Note also that $c_{\rho,\delta}$ is smooth and constant away from a neighborhood of $(x^\beta,t^\beta)$, and $c_{\rho,\delta}(x_{\rho,\delta},t_{\rho,\delta})= u(x_{\rho,\delta},t_{\rho,\delta})$. }

 {Since  $\hat F_{\rho,\delta}(t)= \{   \tilde u(\cdot,t)\ge c_{\rho,\delta}(\cdot,t) \}$, by definition of $\tilde u$ one can check that
\begin{equation}\label{eq:prophatF}
\begin{split}
    \hat F_{\rho,\delta}(t) = \Big\lbrace x\in\R^N : x\in \xi + F(t-\tau^{\rho,\delta}_u) \text{ for some } \xi\in\R^N \text{ with } |\xi|\le \ell^{-1}\left(  \frac{1-c_{\rho,\delta}}\lambda  \right)  \Big\rbrace .
\end{split}	
\end{equation}
For $\rho,\delta$ small enough it holds $x_{\rho,\delta}\notin  F(t_{\rho,\delta}-\tau^{\rho,\delta}_u) $ (from \eqref{eq:defuv1}), and so we introduce  $ w_{\rho,\delta}$  such  that $  x_{\rho,\delta} + w_{\rho,\delta}$ is the projection of $ x_{\rho,\delta}$ on $F(t_{\rho,\delta}-\tau_u^{\rho,\delta})$. In particular, $\xi= -w_{\rho,\delta}$  reaches the max in \eqref{eq:defuv1} for $x=x_{\rho,\delta}$. Also, $|w_{\rho,\delta}|=\ell^{-1}((1-c_{\rho,\delta}(x_{\rho,\delta},t_{\rho,\delta}))/\lambda)$. We define 
\[
\Psi_{\rho,\delta}(x) = x- \ell^{-1}\left( \frac{1-c_{\rho,\delta}(x,t_{\rho,\delta})}\lambda \right)\frac{w_{\rho,\delta}}{|w_{\rho,\delta}|} + w_{\rho,\delta},
\]
which is a $C^2$ diffeomorphism, since $c_{\rho,\delta}$ is bounded away from 1. Moreover, $\Psi_{\rho,\delta}$ is a constant small translation out of a neighborhood of $x^\beta$, converges in $C^2$ to the identity as ${\rho,\delta}\to 0$, and $\Psi_{\rho,\delta}(x_{\rho,\delta})=x_{\rho,\delta}$.
}
From this, define the set  
\begin{equation*}
\check F_{\rho,\delta}(t):=\Psi_{\rho,\delta} (F(t- \tau^{\rho, \delta}_u) - w_{\rho,\delta}).
\end{equation*}
By construction, ${\check F}_{\rho,\delta}(t_{\rho, \delta}) \subseteq  \hat F_{\rho,\delta}(t_{\rho, \delta})$ and $x _{\rho, \delta} \in \partial {\check F}_{\rho,\delta}(t _{\rho, \delta}) \cap \partial  \hat F_{\rho,\delta} (t _{\rho, \delta})$. 
Since  
$\hat F_{\rho,\delta}(t _{\rho, \delta})$  satisfies a uniform external ball condition in $x _{\rho, \delta}$, so does $\check F_{\rho,\delta}(t _{\rho, \delta})$  {(with possibly a different radius)}.  Since $F$ satisfies an internal ball condition uniformly in time  {and $\Psi_{\rho,\delta}$ converges $C^2$ to the identity as ${\rho,\delta}\to0$}, we can assume additionally that  $\check F_{\rho,\delta}(t _{\rho, \delta})$ satisfies  a uniform internal
ball condition for $\rho,\delta$ small enough, with radius depending on $\lambda$.

Finally, defining
\begin{align*}
 (p_{\rho,\delta}, X_{\rho,\delta}) := 
\Big( &D_x(\breve\Phi_{\rho, \delta}(\cdot, t_{\rho, \delta}, y_{\rho, \delta}, s_{\rho, \delta})\circ\Psi_{\rho, \delta})(x_\delta), \\
&D_x^2(\breve\Phi_{\rho, \delta}(\cdot, t_{\rho, \delta}, y_{\rho, \delta}, s_{\rho, \delta})\circ\Psi_{\rho, \delta})(x_\delta) \Big),
\end{align*}
one can check that  {by construction (see \eqref{parajet}) it holds }
$$
( \breve a_{\rho,\delta},p_{\rho,\delta}, X_{\rho,\delta}) \in \PJet^{2,+}_{\hat u(x _{\rho, \delta},t _{\rho, \delta}) \chi_{F(t- \tau^{\rho, \delta}_u  )   }  }(x _{\rho, \delta} + w_{\rho,\delta},t_{\rho,\delta})
$$
Since $\hat u(x _{\rho, \delta},t _{\rho, \delta}) \chi_{F(t- \tau^{\rho, \delta}_u)}$ is a subsolution, we have
\begin{equation}\label{eq:contrad1}
 \breve a_{\rho,\delta} + \psi(|p_{\rho,\delta}|)\G\left(  \kappa_* (x _{\rho, \delta} + w_{\rho,\delta},p_{\rho,\delta}, X_{\rho,\delta},F(t _{\rho, \delta} - \tau^{\rho, \delta}_u )) 
+ \mathtt f(t_{\rho,\delta}) \right)  \le 0.
\end{equation}
Note that  
$$
p_{\rho,\delta}\to Du(x^\beta, t^\beta)\neq 0\,,
$$
as $\rho$, $\delta\to 0$, and thus $|p_{\rho,\delta}|$ is bounded away from zero for  $\rho$ and $\delta$ sufficiently small. Since also $\breve X_{\rho,\delta}$
and hence $X_{\rho,\delta}$ is bounded, the curvature terms $\kappa_* (x _{\rho, \delta},  \breve p_{\rho,\delta},  \breve X_{\rho,\delta}, \check F_{\rho, \delta}(t _{\rho, \delta}  )) $ are uniformly bounded from above and below as $\rho,\delta\to0$. Thus, $\G$ is uniformly continuous and by Lemma~\ref{lm:compkappa}
we deduce from \eqref{eq:contrad1} that 
\begin{equation}\label{eq:contrad3}
 \breve a_{\rho,\delta} + \psi(| \breve p_{\rho,\delta}|) \G\left( \kappa_* (x _{\rho, \delta},  \breve p_{\rho,\delta},  \breve X_{\rho,\delta}, \check F_{\rho, \delta}(t _{\rho, \delta}  )) 
+ \mathtt f(t_{\rho,\delta}) \right) \le \omega(\rho,\delta),
\end{equation}
where $\omega$ is a modulus of continuity that depends on $\beta$, and $ \omega(\rho,\delta)\to 0$ as $\rho,\, \delta\to 0$. 
%Note that $\omega$ depends on $\beta$ via $|x^\beta-y^\beta|$ only.
Analogously, from \eqref{parajet2}, \eqref{eq:relparjet} and since $\hat v(y_{\rho,\delta},s_{\rho,\delta})\chi_{ G(t-\tau_v^{\rho,\delta})}$ is a supersolution,  we also have
\begin{equation}\label{eq:contrad4}
 \breve a_{\rho,\delta} -2\e +\psi( | \breve p_{\rho,\delta}|) \G\left(  \kappa^* (y _{\rho, \delta} ,  \breve p_{\rho,\delta},  \breve Y_{\rho,\delta}, \check G_{\rho, \delta}(s _{\rho, \delta})) + \mathtt f(s_{\rho,\delta})\right)  \ge \omega(\rho,\delta)
\end{equation}
for a suitable set $\check G_{\rho, \delta}(s _{\rho, \delta}))$ such that 
$$\hat F(t _{\rho, \delta}  ) + (y _{\rho, \delta} - x _{\rho, \delta}) \subseteq \check G_{\rho, \delta}(s _{\rho, \delta}))
\text{ and }  
\partial (\check F _{\rho, \delta}(t _{\rho, \delta}  )+  (y _{\rho, \delta} - x _{\rho, \delta})) \cap \partial \check G _{\rho, \delta}(s _{\rho, \delta})) = \{y _{\rho, \delta}\}.
$$
By the  {equation} above, \eqref{eq:relparjet} and Lemma~\ref{lm:compC'} we get 
$$
\kappa_* (x _{\rho, \delta},  \breve p_{\rho,\delta},  \breve X_{\rho,\delta}, \check F_{\rho, \delta}(t _{\rho, \delta}  )) \ge
\kappa^* (y _{\rho, \delta} ,  \breve p_{\rho,\delta},  \breve Y_{\rho,\delta}, \check G_{\rho, \delta}(s _{\rho, \delta})),
$$ 
and thus by \eqref{eq:contrad3} and \eqref{eq:contrad4} we arrive at 
\begin{equation}\label{eq:end-contr}
    \begin{split}
        -2\e +  \psi( | \breve p_{\rho,\delta}| ) \Big[   \G\left( \kappa_* (x _{\rho, \delta},  \breve p_{\rho,\delta},  \breve X_{\rho,\delta}, \check F_{\rho, \delta}(t _{\rho, \delta}  ))  +\mathtt f(s_{\rho,\delta})\right) \\
 -  \G\left( \kappa_* (x _{\rho, \delta},  \breve p_{\rho,\delta},  \breve X_{\rho,\delta}, \check F_{\rho, \delta}(t _{\rho, \delta}  ))  +\mathtt f(t_{\rho,\delta})\right)  \Big] 
 \ge 2\omega(\rho,\delta).
    \end{split}
\end{equation}
 \noindent\textbf{Step 2:} We now pass to the limit $\rho,\delta\to 0$ then $\beta\to+\infty.$
Recalling that $(x^\beta,y^\beta)\to (x_0,y_0)$ as $\beta\to+\infty$, we distinguish two cases.\\
\noindent\textbf{Case 1: }Assume that $x_0\neq y_0$. In this case, $|x^\beta-y^\beta|$ is uniformly bounded from below, and thus  the term $\kappa_* (x _{\rho, \delta},  \breve p_{\rho,\delta},  \breve X_{\rho,\delta}, \check F_{\rho, \delta}(t _{\rho, \delta}  ))$ is uniformly bounded in $\rho,\delta,\beta$.
%In turn, the modulus of continuity $\omega$ appearing in \eqref{eq:end-contr} does not depend on $\beta$. 
Therefore, the continuity of $\G$ is uniform as $\rho,\delta,\beta$ vary, and \eqref{eq:end-contr} implies 
\begin{equation}\label{eq:problemsf?}
     \tilde \omega\left(| \mathtt f(s^\beta)-\mathtt f(t^\beta)| \right) + \omega(\rho,\delta) \ge \e,
\end{equation}
for a  modulus of continuity $\tilde \omega$ such that $\tilde\omega (r)\to 0$ as  $r\to 0$.
We  pass to the limit $\rho,\delta\to0$ then $\beta\to +\infty$   to arrive at  a contradiction.

\noindent\textbf{Case 2:} It holds $x_0 = y_0$. In this case, we recall that 
\[ p_{\rho,\delta}\to Du(x^\beta, t^\beta)=\alpha \ell'(|x^\beta-y^\beta|), \quad \text{ as } \rho,\delta\to0.  
\]
Recall that the set $ \check F_{\rho, \delta}(t _{\rho, \delta}  )$ satisfies a uniform internal and external ball condition of radius $|x^\beta-y^\beta|$. In particular
\[  |\kappa_* (x _{\rho, \delta},  \breve p_{\rho,\delta},  \breve X_{\rho,\delta}, \check F_{\rho, \delta}(t _{\rho, \delta}  )) | \le \overline c(|x^\beta-y^\beta|). \]
Therefore, equation \eqref{eq:end-contr} implies 
\[
\psi( | p_{\rho,\delta}|+\omega(\rho,\delta))\  \G\left(  \overline c(|x^\beta-y^\beta|+\|\mathtt f\|_\infty)\right) + 2\omega(\rho,\delta) \ge 2\e.
\]
Sending $\rho,\delta\to 0$ we get 
\[
c_\psi |\ell'(|x^\beta-y^\beta|)|  \,  \G\left(  \overline c(|x^\beta-y^\beta|)\right) \frac{\G\left(  \overline c(|x^\beta-y^\beta|+\|\mathtt f\|_\infty)\right)}{\G\left(  \overline c(|x^\beta-y^\beta|)\right)}\ge 2\e.
\]
recalling the properties of $\ell$ (see Definition \ref{family F}), we arrive at a contradiction sending $\beta\to +\infty$.
\end{proof}

\subsection*{Acknowledgements}
The author wishes to thank professors A. Chambolle, M. Morini, and M. Novaga for many helpful discussions and comments. 
 {The author wishes to thank the anonymous referee fo the careful reading of the manuscript and his
comments, which helped improve the paper.} 
The majority of this work was carried out during the author's PhD at Paris-Dauphine University. The author wishes to express gratitude for the warm and convivial atmosphere experienced there.

The author  has received partial funding from the European Union's Horizon 2020 research and innovation program under the Marie Skłodowska-Curie grant agreement No 945332.
The author is funded by the European Union: the European Research Council (ERC), through StG ``ANGEVA'', project number: 101076411. Views and opinions expressed are however those of the authors only and do not necessarily reflect those of the European Union or the European Research Council. Neither the European Union nor the granting authority can be held responsible for them.
\subsection*{Data Availability Statement}
Data sharing is not applicable to this article as no datasets were generated or analyzed during the
current study
\subsection*{Conflict of interest}
The authors declare that they have no conflict of interest.

\printbibliography
\end{document}